\documentclass[12pt, a4paper]{article}

\usepackage{a4,amsmath,amssymb, amsthm, latexsym, color, graphicx,url}
\usepackage{tikz}
\usetikzlibrary{decorations.pathreplacing,decorations.markings}
\tikzset{
  on each segment/.style={
    decorate,
    decoration={
      show path construction,
      moveto code={},
      lineto code={
        \path [#1]
        (\tikzinputsegmentfirst) -- (\tikzinputsegmentlast);
      },
      curveto code={
        \path [#1] (\tikzinputsegmentfirst)
        .. controls
        (\tikzinputsegmentsupporta) and (\tikzinputsegmentsupportb)
        ..
        (\tikzinputsegmentlast);
      },
      closepath code={
        \path [#1]
        (\tikzinputsegmentfirst) -- (\tikzinputsegmentlast);
      },
    },
  },
  mid arrow/.style={postaction={decorate,decoration={
        markings,
        mark=at position .5 with {\arrow[#1]{stealth}}
      }}},
  start arrow/.style={postaction={decorate,decoration={
        markings,
        mark=at position .25 with {\arrow[#1]{stealth}}
      }}},
}
\usetikzlibrary{arrows.meta}
\usetikzlibrary{decorations.markings}
\usetikzlibrary{calc}

\usepackage{authblk}
\usepackage{fullpage}
\usepackage{enumitem}
\usepackage{xcolor}
\newtheorem{theorem}{Theorem}[section]
\newtheorem{proposition}[theorem]{Proposition}
\newtheorem{lemma}[theorem]{Lemma}
\newtheorem{corollary}[theorem]{Corollary}

\newtheorem{conjecture}{Conjecture}
\theoremstyle{definition}
\newtheorem{example}[theorem]{Example}
\newtheorem{definition}[theorem]{Definition}
\newtheorem{remark}[theorem]{Remark}
\newtheorem{construction}[theorem]{Construction}

\newcommand{\E}{\overrightarrow{E}}

\title{Graph labellings and external difference families}
\author[1]{Gavin Angus}
\author[2] {Sophie Huczynska}
\author[3] {Struan McCartney}
\affil[1]{School of Computing Science, University of Glasgow, Glasgow G12 8RZ, Scotland, UK; email: gavin.angus04@gmail.com}
\affil[2]{School of Mathematics and Statistics, University of St Andrews, St Andrews, KY16 9SS, Scotland, UK; email: sh70@st-andrews.ac.uk}
\affil[3]{School of Mathematics and Statistics, University of St Andrews, St Andrews, KY16 9SS, Scotland, UK: email: sm444@st-andrews.ac.uk}
\date{MSC codes: 05B10, 05C78, 94A60}
\begin{document}
\maketitle

\begin{abstract}
Digraph-defined external difference families were recently introduced as a natural generalization of several well-studied combinatorial objects motivated by cryptography (e.g. external difference families (EDFs) and circular external difference families (CEDFs)).  In this paper, we develop a systematic framework for using various types of vertex-labellings for graphs and digraphs to create digraph-defined external difference families.   The approach is to combine suitable vertex-labellings (generalizations of $\alpha$-valuations, namely near $\alpha$-valuations and oriented near $\alpha$-valuations) with a graph blow-up technique. Many new families are produced, including the first explicit construction for an infinite family of $2$-CEDFs, achieving all parameter sets for $(n,m,l;1)$-$2$-CEDFs with $m \equiv 0 \mod 4$ sets. Further, new results arise for graph labellings themselves (e.g. cyclotomy-based near $\alpha$-valuations for a family of trees without $\alpha$-valuations, and an $\alpha$-valuation for sun graphs).
\end{abstract}

\section{Introduction}
External difference families (EDFs) are much-studied combinatorial objects, initially motivated by applications in information security \cite{Cramer, Oga}.  In its original form, an EDF consists of a finite collection of equal-sized disjoint subsets $A_1,\ldots, A_m$ of a finite group $G$, such that the multiset union $\cup_{i \neq j} \{x-y: x \in A_i, y \in A_j\}$ comprises a constant number of occurrences of each non-identity element of $G$. Several variants of these have been investigated, in particular strong external difference families (SEDFs) which correspond to a stronger security model than the original setup, and circular difference families (CEDFs) which were introduced in relation to non-malleable codes \cite{VeiSti}. EDFs and their variants are essentially collections of sets with conditions on the multiset of differences occurring between the elements of certain pairs of sets in the collection.  Recently, the concept of digraph-defined EDFs was introduced, which provides a framework for all of these (and more) by using directed graphs to specify how the multiset of pairwise differences between sets is defined in each case \cite{HucJefMcC}.  For example, the original EDF is defined using the complete undirected graph (with each edge viewed as being a pair of directed edges), while a CEDF is defined by a cycle oriented in a clockwise direction.

Graceful labellings (also called $\beta$-valuations), and in particular the special case of $\alpha$-valuations, are vertex labellings of graphs which have been used in various contexts to produce difference structures.  Recently they were used by Stinson and couathors to create 1-CEDFs and SEDFs \cite{PatStiCEDF, PatSti1, VeiSti}.  A variation of graceful labelling was used by Buratti in \cite{Bur} to create strong difference families (note that these differ from SEDFs).

In this paper, we identify those properties of vertex-labellings of graphs and digraphs of use for constructing difference structures, with a focus on digraph-defined external difference families. Previous related work \cite{PatStiCEDF} used certain vertex-labellings of even-length cycles; inspired by this, we build a systematic framework for the use of graph-labellings in a general context, illuminating which aspects of the valuation are used.  We demonstrate that more general types of labelling possess the necessary properties, e.g. near $\alpha$-valuations \cite{Cahit,Grannell,El-Zanati} and oriented near $\alpha$-valuations \cite{BloHsu}.  An advantage of these is that they can exist for types of graph which are known not to possess an $\alpha$-valuation. We develop new results about them, and demonstrate how infinite families of digraph-defined EDFs can be constructed using them.  In the final section we give constructions of digraph families possessing a specific ``natural orientation" (unidirectional for paths and cycles, and left-to-right/bottom-to-top for grids).  We provide the first explicit construction for infinite families of $2$-CEDFs, which covers all possible parameter sets for $(n,m,l;1)$-$2$-CEDFs with $m \equiv 0 \mod 4$ sets. Since a $2$-CEDF with an odd number of sets is a re-ordered $1$-CEDF, this leaves only the case when $m \equiv 2 \mod 4$. 

\section{Background}
Throughout, for a graph $G=(V,E)$, we denote by $V(G)$ the vertex-set of $G$ and by $E(G)$ the edge-set of $G$. Our graphs are finite and simple with no loops or isolated vertices. When we consider directed graphs, we consider those which are oriented, i.e. for distinct $u,v$ in the vertex set, at most one of $(u,v)$ or $(v,u)$ is in the directed edge set.

\subsection{$\beta$-valuations and digraph-defined EDFs}

By a vertex-labelling of a graph, we mean any one-to-one map of its vertices into the set of integers $\{0,1,\ldots,n\}$ ($n \in \mathbb{N}$).
The following definition of a particular type of vertex-labelling is given by Rosa in \cite{Rosa}.

\begin{definition}
A $\beta$‐valuation of a graph $G$ with $n$ edges is a one‐to‐one map of the vertices into the set of integers $\{0,1,\ldots,n\}$, such that, if we label each edge of $G$ by the absolute value of the difference of the labels of the corresponding vertices, then the resulting edge labels are precisely the elements of the set $\{1,2,\ldots,n\}$. (This is also known as a graceful labelling.)

Formally, for a graph $G$, we define a $\beta$-valuation $b$ of $G$ as an injective function $b : V(G) \rightarrow \{0,\ldots,n\}$
such that the following multiset equation holds:
\[\bigcup_{\{u,v\}\in E(G)} \{ \mid b(v)-b(u) \mid \}= \{1, \ldots n\}\]
\end{definition}
The Graceful Tree Conjecture was first proposed by Rosa in \cite{Rosa}. It conjectures that all trees have a $\beta$-valuation, and is still open in general. Various classes of trees are known to have $\beta$-valuations \cite{Rosa}, including caterpillars (a caterpillar is a tree such that, when all of its pendant vertices are removed, the resulting graph is a path) and rooted symmetric trees \cite{Rob}.  Graphs which are not trees can also have $\beta$-valuations: e.g. the cycle $C_m$ has a graceful labelling precisely if $m$ is $0$ or $3$ modulo $4$ \cite{Rosa}.

As we label the vertices with distinct values from the set of integers $\{0,1,\ldots,n\}$, for a graph to possess a $\beta$-valuation it is necessary that $ n +1 \geq \mid V(G) \mid.$ This means that a $\beta$-valuation is possible only for $n$-edge graphs which have at most $n+1$ vertices. A tree is an extremal example of this (with $|E(G)|=n$ and $V(G)=n+1$). An example of a graph that does not have a $\beta$-valuation is $P_1 \cup P_1$, the union of two paths of length 1.

The following definition is introduced in \cite{HucJefMcC}.

\begin{definition}\label{def:H-defined EDF}
Let $G$ be a group of order $n$, and let $\mathcal{A}=(A_0,\ldots,A_{m-1})$ be an ordered collection of disjoint subsets of $G$, each of size $l$.  Let $H$ be a labelled digraph on $m$ vertices $\{0,1,\ldots,m-1\}$ and let $\E(H)$ be the set of directed edges of $H$.  Then $\mathcal{A}$ is said to be an $(n,m,l,\lambda; H)$-EDF if the following multiset equation holds:
$$ \bigcup_{(i,j) \in \E(H)} \Delta(A_j,A_i) = \lambda (G \setminus \{0\}).$$
We will call such a structure a \emph{digraph-defined} EDF.  If we wish to emphasise $H$, we will call it an $H$-defined EDF.
\end{definition}

The following type of digraph-defined EDF, introduced by Veitch and Stinson in \cite{VeiSti}, has attracted particular attention:
\begin{definition}
Let $G$ be a group of order $n$. Suppose $m>1$ and $1 \leq c \leq m-1$.   A family of disjoint $l$-sets $\{A_1, \ldots, A_m\}$ in $G$ is an $(n, m, l,\lambda)$-$c$-CEDF if the following multiset equation holds:
$ \bigcup_{i=0}^{m-1} \Delta(A_{i+c \mod m}, A_i)=\lambda (G \setminus \{0\})$.
\end{definition}

If $c=1$ then $c$ is usually omitted from the notation.  The following was proved in \cite{HucJefMcC}:

\begin{lemma}\label{lemma:CEDF}
Let $\mathcal{A}=\{A_1, \ldots, A_m\}$ be a disjoint collection of $l$-subsets of a group $G$. Suppose $\mathcal{A}$ is a $c$-CEDF.  Then, if $\gcd(c,m)=d$, the multiset in the definition may be written as a disjoint union of $d$ multiset unions, each involving $m/d$ sets, as follows (where all indices are taken modulo $m$):
$$ \bigcup_{i=0}^{m-1} \Delta(A_{i+c}, A_i)=\bigcup_{j=0}^{d-1} \left(\bigcup_{i=0}^{(m/d-1)} \Delta(A_{(i+1)c+j},A_{ic+j}) \right)=\lambda(G \setminus \{0\})$$
\end{lemma}

We can use vertex-labellings on graphs to obtain digraphs and corresponding EDFs.

\begin{definition}
Let $G$ be a graph equipped with a $\beta$-valuation $b$.  Denote by $G^b$ the digraph obtained from $G$ by orienting each edge towards the larger of its two vertex labels under the $\beta$-valuation $b$.  We will call this its \emph{natural orientation}.
\end{definition}
We will, more generally, apply this concept to any vertex-labelling $b$ of a graph $G$.

The definition of the natural orientation immediately implies the following:
\begin{lemma}\label{betatodirect}
    Let $G$ be a graph which has $\beta$-valuation $b$. Then the following multiset equation holds:
    \[\bigcup_{\{u,v\}\in E(G)} \{\mid b(v)-b(u) \mid \} = \bigcup_{(x,y)\in \overrightarrow{E}(G^b)} \{(b(y)-b(x) )\} =\{1, \ldots n\}\]
\end{lemma}

We can immediately use $\beta$-valuations as a way to construct $G^b$-defined EDFs in cyclic groups, with $l=\lambda=1$. Throughout what follows, we let the elements of $\mathbb{Z}_{n+1}$ be $\{0,1,\ldots,n\}$, with the ordering $0<1< \cdots <n$.  Where appropriate, we will identify the elements of $\mathbb{Z}_{n+1}$ with the corresponding integers in $\{0,\ldots,n\}$.

\begin{theorem}
    Let $G$ be a graph with $n$ edges, with $\beta$-valuation $b$.  Then there exists an $(n,m,1,1;G^b)$-EDF in $\mathbb{Z}_{n+1}$.
\end{theorem}
\begin{proof}
For each $v \in V(G)$, let $A_v=\{b(v)\}$, where these are viewed as disjoint subsets of $\mathbb{Z}_{n+1}$. Then:
\[\bigcup_{(u,v)\in E(G^b)} \Delta(A_v,A_u)= \bigcup_{(u,v)\in E(G^b)} \{b(v)-b(u)\}=\{1, \ldots n\}=\mathbb{Z}_{n+1}\backslash\{0\}.\]
Hence $(A_{v_1},A_{v_2},\ldots,A_{v_{m}})$ forms a $(n,m,1,1;G^b)$-EDF in $\mathbb{Z}_{n+1}$.
\end{proof}

\subsection{Graph blow-up constructions}

In \cite{PatStiCEDF}, a process is presented which allows larger graphs to be obtained from smaller ones, in a way which allows certain graph labellings to be carried-over to the larger graphs. (In \cite{PatStiCEDF}, these labellings were special $\beta$-valuations called $\alpha$-valuations; we will apply this to a more general class of labellings).  We refer to the process of obtaining the new graph as a ``blow-up construction".  For completeness, and to aid in our subsequent constructions, we introduce some new notation and provide full details of the relevant proofs.

\begin{definition}
Let $G$ be a graph. We say that a \emph{blow-up} of $G$ is a graph obtained from $G$ as follows: 
\begin{itemize}
\item each vertex $s \in V(G)$ is replaced by a set of vertices $V_s$ (of size at least $1$).
\item for $s,t \in V(G)$, each vertex in $V_s$ is adjacent to each vertex in $V_t$ precisely if $s$ and $t$ are adjacent in $G$.
\end{itemize}
If there is some $k$ such that $|V_s|=k$ for all $s \in V(G)$, the blow-up is said to be \emph{balanced}. 
\end{definition}

\begin{remark}
If a blow-up is balanced then it is uniquely determined by the size of the vertex blow-up sets.   We refer to this as \emph{the balanced blow-up of $G$ by a factor of $k$}. 
\end{remark}

\begin{definition}
The lexicographic product $G \cdot H$ of two graphs $G=(V(G),E(G))$ and $H=(V(H),E(H))$ is a graph on the vertices $V(G)\times V(H)$, where two vertices $(x,y)$ and $(u,v)$ are adjacent if and only if:
\[\{x,u\} \in E(G) \mbox{, or } x =u \textit{ and } \{y,v\} \in E(H).\]
\end{definition}

\begin{lemma}\label{lexi}
Let $G$ be a graph and denote by $K^c_l$ the empty graph on $l$ vertices. The lexicographic product $G \cdot K^c_l$ is the balanced blow-up of $G$ by a factor of $l$.
\end{lemma}
\begin{proof}
Let $H = G \cdot K_l^c$. Its vertex set is given by $V(H) = V(G) \times V(K_l^c)$. To see that the first blow-up condition is satisfied, note that $V(H)$ is constructed by taking each vertex $v \in V(G)$ and replacing it by a set ($V_v$ say) of $l$ vertices: 
    \[\{(v,x_0),(v,x_1),\ldots, (v,x_{l-1})\}.\]
Since all the sets are of the same size, the balanced condition holds. 
The edge set of $H$ is given by $E(H)=\{\{(x,y),(u,v) : \{x,y\} \in E(G) \text{ or } x = u \text{ and }\{y,v\} \in E(K_l^c)\}$. As $E(K_l^c) = \emptyset$, the second condition is never satisfied, so the edges of $H$ are given by:
	\[E(H)=\{(x,y),(u,v) : \{x,u\} \in E(G)\}.\]
In other words, the edges of $H$ occur precisely between all pairs of vertices such that one vertex is in $V_x$ and the other is in $V_u$, where $x$ and $u$ are adjacent in $G$.  Hence the second condition for a blow-up is satisfied.   
\end{proof}
\begin{lemma} \label{blow-upblow-up}
    Let $G,H$ and $F$ be graphs such that $G$ is a blow-up of $H$ and $H$ is a blow-up of $F$.  Then $G$ is a blow-up of $F$.
\end{lemma}
\begin{proof}
Starting with $F$, apply the two blow-ups in turn to obtain $G$; we claim that the result is also a blow-up. In obtaining $H$ from $F$, each vertex $u \in V(F)$ is replaced by a set of vertices $U_u$. Then to obtain $G$ as a blow-up of $H$, each vertex $v \in V(H)$ is replaced by a set of vertices $V_v$.  Hence in $G$, each vertex $u \in V(F)$ has been replaced by the set $\cup_{v \in U_u}V_v$ (denote this by $S_u$), and the first blow-up condition holds. Let $a \in S_u$ and $b \in S_v$ be arbitrary.  Then $a \in V_c$ for some $c \in U_u$ and $b \in V_d$ for some $d \in U_v$.  There is an edge between $a$ and $b$ if and only if there is an edge between $c$ and $d$ in $H$, if and only if there is an edge between $u$ and $v$ in $F$.  This gives both parts of the blow-up condition.
\end{proof}

The following specific application to bipartite graphs will be important later.  For a bipartite graph $G$ with vertex bipartition $V(G)=S \cup T$, we denote by $G_{k_1,k_2}$ the blow-up of $G$ in which $|V_s| = k_1$ for all $s\in S$ and $|V_t| = k_2$ for all $t\in T$ (here the bipartition and the set sizes are sufficient to determine the blow-up uniquely).  
\begin{corollary}\label{lexiblow}
Let $G$ be a graph with vertex partition $V(G)=S \cup T$.  Then the graph obtained by first blowing-up the vertices in $S$ by a factor of $l$, then blowing-up the vertices in $T$ by a factor of $l$, is precisely the balanced blow-up of $G$ by a factor of $l$,  i.e. the lexicographic product of $G \cdot K_l^c$.
\end{corollary}
\begin{proof}
The first step yields the graph $G'=G_{l,1}$; this is bipartite, with bipartition $V(G')=(\cup_{s \in S}V_s) \cup T$.  The second step yields the graph $G'_{1,l}$; by definition, the resulting graph is the balanced blow-up of $G$ by a factor of $l$. The last part follows by Lemma \ref{lexi}.
\end{proof}

\section{Graceful labellings and corresponding EDFs}

Our strategy in this paper will be to show how various graph-labellings can be combined with the blow-up construction to obtain new digraph-defined EDFs.  However, the $\beta$-valuation condition alone is not sufficient to enable this approach; an extra condition is needed.  A ``proof of concept" of this type of approach, used to obtain CEDFs, is demonstrated in \cite{PatStiCEDF}.  We will use a more general class of valuation called near $\alpha$-valuations; however we first introduce the better-known special case of $\alpha$-valuations.

\subsection{$\alpha$-valuations of graphs}

\begin{definition}
An $\alpha$-valuation $b$ of a graph $G$ with $n$ edges is a $\beta$‐valuation which satisfies the following extra condition:
there exists some $x$ with $0 \leq x \leq n$ such that for each edge $\{u,v\} \in E(G)$: 
\[b(v) \leq x < b(u) \text{ or } b(u) \leq x < b(v).\]
\end{definition}

Rosa notes in \cite{Rosa} that a graph $G$ with an $\alpha$-valuation must be bipartite; a bipartition of $V(G)$ is given by the sets $V_{small}$ and $V_{large}$, where the former consists of those vertices with
labels at most $x$ and the latter consists of vertices with labels larger than $x$.

Various natural families of graphs are known to have $\alpha$-valuations. For example, the following $\alpha$-valuation for paths is adapted from \cite{BloHsu}. 
\begin{theorem}\label{pathalpha}
Let $P_m$ denote the path with $m$ vertices, with $V(P_m)=\{v_1,\ldots,v_m\}$.  Then an $\alpha$-valuation is given by:
\[a(v_i) = \begin{cases} 
          \frac{i-1}{2} & i \textit{ odd} \\
          m- \frac{i}{2} & i \textit{ even} \\
\end{cases}
\]
\end{theorem}
In \cite{Rosa}, Rosa showed that all caterpillars have $\alpha$-valuations and gave a diagrammatic proof, a result which has since been extended in various ways, e.g. by amalgamating paths and caterpillars as in \cite{Barrientos}.  Non-trees can also possess $\alpha$-valuations. In \cite{Rosa}, an $\alpha$-valuation is given for the complete bipartite graph $K_{p,q}$.
\begin{theorem}
Let $K_{p,q}$ denote the complete bipartite graph with $V(G)=\{v_1,\ldots,v_p\} \cup \{u_1,\ldots, u_q\}$.
Then an $\alpha$-valuation $a$ is given by: $a(v_i)=i-1$ and $a(u_j)=jp$.
\end{theorem}

The cycle $C_m$ has an $\alpha$-valuation precisely if $m \equiv 0 \mod 4$; the following is from \cite{PatStiCEDF}.
\begin{theorem}\label{cyc}
Let $m \equiv 0 \mod 4$ and let $C_m$ denote the cycle of length $m$, with $V(C_m)=\{v_1,\ldots,v_m\}$.  Then an $\alpha$-valuation $a$ is given by:
\[a(v_i) = \begin{cases} 
          \frac{i-1}{2} & i \textit{ odd} \\
          m +1- \frac{i}{2} & i \textit{ even, }i \leq \frac{m}{2} \\
          m - \frac{i}{2} & i \textit{ even, }i > \frac{m}{2} \\
       \end{cases}
    \]
\end{theorem}

References for a range of graphs with $\alpha$-valuations is given in the dynamic survey \cite{Gal}.

While it is possible that all trees may admit $\beta$-valuations, it is already known that not all trees admit $\alpha$-valuations. Denote by $S_{m,l}$ the graph formed by taking the star $K_{m,1}$ and replacing each edge by a path of length $l$.  Then a tree of smallest size which does not possess an $\alpha$-valuation is given in \cite{Rosa}; it is the 7-vertex graph $S_{3,2}$.  However as shown in \cite{Barrientos}, graphs which contain this as a subgraph may still have $\alpha$-valuations.  

\begin{figure}
    \centering
    \begin{tikzpicture}[baseline=4mm, scale=0.9]
    \begin{scope}[radius=1mm]

    \fill(1.5,0.5) circle;
    \fill(0,-2.5) circle;
    \fill(3,-2.5) circle; 
    \fill(-1.5,-2.5) circle; 
    \fill(4.5,-2.5) circle; 
    \fill(0,-1) circle;
    \fill(3,-1) circle; 
    \fill(-1.5,-1) circle; 
    \fill(4.5,-1) circle; 
    \end{scope}
    \begin{scope}[thick]
    \draw (0,-1)--(0,-2.5);
    \draw (0,-1)--(1.5,0.5);
    \draw (3,-1)--(3,-2.5);
    \draw (3,-1)--(1.5,0.5);
    \draw (-1.5,-1)--(-1.5,-2.5);
    \draw (-1.5,-1)--(1.5,0.5);
    \draw (4.5,-1)--(4.5,-2.5);
    \draw (4.5,-1)--(1.5,0.5);
    \node at (1.5,0.75) {$x_0$};
    \node at (-1.75,-2.75) {$z_0$};
    \node at (-1.75,-1.25) {$y_0$};
    \node at (-0.25,-2.75) {$z_1$};
    \node at (-0.25,-1.25) {$y_1$};
    \node at (3.5,-2.75) {$z_{m-2}$};
    \node at (3.5,-1.25) {$y_{m-2}$};
    \node at (5,-2.75) {$z_{m-1}$};
    \node at (5,-1.25) {$y_{m-1}$};
    \node at (1.5,-2.75) {$\ldots$};
    \node at (1.5,-1.25) {$\ldots$};
    \end{scope}
    \end{tikzpicture}
    \caption{$S_{m,2}$}
    \label{fig:S_{m,2}}
\end{figure}

More generally, Rosa presented a collection of trees without $\alpha$-valuations; this is the set $S(2,4)$ of trees $T$ defined as follows. Recall that the \emph{diameter} of a graph is the maximum length of any shortest-path between two distinct vertices of the graph.

\begin{definition}\label{def:S(2,4)}
Let $T$ be a tree.  Then $T$ is in the set $S(2,4)$ if it possesses the following three properties:
\begin{itemize}
\item $T$ has diameter $4$; 
\item removing all leaf vertices and associated edges from $T$ yields a new graph $T'$ which is not a path; 
\item removing all leaf vertices and associated edges from $T'$ yields a single vertex.  
\end{itemize}   
\end{definition}

These graphs were further investigated by Huang, Kotzig and Rosa in \cite{HuKoRo}.

\subsection{Near $\alpha$-valuations of graphs}

In fact, for use in the blow-up construction to obtain digraph-defined EDFs, a weaker additional condition on our $\beta$-valuation is sufficient.  For a vertex $v$ in a graph, denote by $N(v)$ the neighbourhood of $v$, i.e. the set of all vertices adjacent to $v$ in the graph.

\begin{definition}
A near $\alpha$-valuation $b$ of a graph $G$ with $n$ edges is a $\beta$‐valuation which satisfies the following extra condition: there is a partition of $V(G)$ into $V_{small}$ and $V_{large}$ such that
\begin{itemize}
\item for each $v \in V_{small}, b(v) < b(u)$ for all $u \in N(v)$, and 
\item for each $v \in V_{large}, b(v) > b(u)$ for all $u \in N(v)$.
\end{itemize}
In other words, the vertices of $G$ can be partitioned into two sets $V_{large}$ and $V_{small}$ such that for each $v \in V_{large}$, its label is larger than the label of any of its neighbours, and for each $v \in V_{small}$, its label is smaller than the label of any of its neighbours.
\end{definition}

This concept was introduced independently in \cite{Cahit}, \cite{Grannell} and \cite{El-Zanati}, under different names.  When we apply the natural orientation, this condition is equivalent to requiring that every vertex in the digraph is either a source or a sink. This is called a source-sink orientation and has been studied by various authors. Digraphs possessing a source-sink orientation are precisely those which avoid the directed two-edge path as a subgraph.

The following result summarizes some key properties from the literature. 
\begin{theorem}\label{thm:nearprops}
\begin{itemize}
\item[(i)] For a $\beta$-valuation $b$ of a graph $G$, $b$ is a near $\alpha$-valuation precisely if every edge of $G$ is of the form $\{u,v\}$ where $u \in V_{small}$, $v \in V_{large}$ and $b(u)<b(v)$. In particular, if a graph $G$ admits a near $\alpha$-valuation, then it is bipartite.
\item[(ii)] Every $\alpha$-valuation of a graph $G$ is a near $\alpha$-valuation of $G$.
\item[(iii)] There exist graphs which admit near $\alpha$-valuations but do not admit $\alpha$-valuations.
\end{itemize}
\end{theorem}
\begin{proof}
i) Let $b$ be a $\beta$-valuation of $G$. Suppose $b$ is a near $\alpha$-valuation with $V(G)=V_{small} \cup V_{large}$.  Let $\{u,v\}$ be an edge in $G$.  If $u\in V_{small}$ then as $v \in N(u)$ we have $b(u)<b(v)$; but then $u \in N(v)$ so $v \in V_{large}$.  Similarly if $u \in V_{large}$, $b(u)>b(v)$ and so $v \in V_{small}$. Conversely, suppose every edge of $G$ is of the form $\{u,v\}$ where $u \in V_{small}$, $v \in V_{large}$ and $b(u)<b(v)$.  Let $u \in V_{small}$; then for any $y \in N(u)$, $\{u,y\} \in E(G)$ and so $y \in V_{large}$ with $b(u)<b(y)$. Similarly if $u \in V_{large}$; then for any $y \in N(u)$, $y \in V_{small}$ with $b(u)>b(y)$, so $b$ is a near $\alpha$-valuation.\\
ii) Let $a$ be an $\alpha$-valuation of $G$. Then there is an integer $x$ such that for each $\{u,v\} \in E(G)$, $a(v) \leq x <a(u)$ or $a(u) \leq x <a(v)$. Define $V_{small}=\{v \in E(G): a(v) \leq x\}$ and $V_{large}=\{ v \in E(G): a(v)>x\}$.  This partitions $V(G)$, and each edge of $G$ is of the form $\{s,t\}$ with $s \in V_{small}$, $t \in V_{large}$ and $a(s) \leq x < a(t)$.  Hence this is a near $\alpha$-valuation.\\
iii) Below is a near $\alpha$-valuation for Rosa's graph $S_{3,2}$ which has no $\alpha-$valuation  \cite{Rosa}:
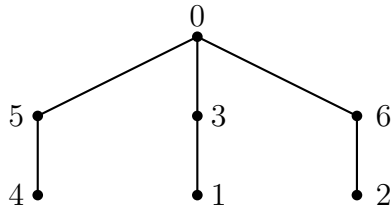
\begin{center}
    \begin{tikzpicture}[baseline=4mm, scale=0.7]
    \begin{scope}[radius=1mm]

    \fill(1.5,0.5) circle;
    \fill(-1.5,-2.5) circle; 
    \fill(4.5,-2.5) circle; 
    \fill(-1.5,-1) circle; 
    \fill(4.5,-1) circle; 
    \fill(1.5,-1) circle;
    \fill(1.5,-2.5) circle;
    \end{scope}
    \begin{scope}[thick]
    \draw (-1.5,-1)--(-1.5,-2.5);
    \draw (-1.5,-1)--(1.5,0.5);
    \draw (1.5,-1)--(1.5,-2.5);
    \draw (1.5,-1)--(1.5,0.5);
    \draw (4.5,-1)--(4.5,-2.5);
    \draw (4.5,-1)--(1.5,0.5);
    \node at (1.5,0.9) {$0$};
    \node at (-1.9,-2.5) {$4$};
    \node at (-1.9,-1) {$5$};
    \node at (5,-2.5) {$2$};
    \node at (5,-1) {$6$};
    \node at (1.9,-2.5) {$1$};
    \node at (1.9,-1) {$3$};
    \end{scope}
    \end{tikzpicture}
    \end{center}
\end{proof}

Similarly to the Graceful Tree Conjecture, there is a (still open) conjecture that every tree admits a near $\alpha$-valuation (this is first stated in \cite{El-Zanati}).

\begin{conjecture}\label{nearalphatree}
Every tree has a near $\alpha$-valuation.
\end{conjecture}

\begin{remark}\label{20nearalphatree}
In \cite{Grannell}, the conjecture is tested computationally for up to 20 vertices: it is confirmed that every tree with at most 20 vertices has a near $\alpha$-valuation.
\end{remark}

In \cite{El-Zanati}, El-Zanati et al exhibit two types of graphs that admit near $\alpha$-valuations but not $\alpha$-valuations. The first is $S_{m,2}$; recall that this is the graph defined by replacing each edge of $K_{m,1}$ by a path of length $2$.  The second is defined as follows:  denote by $B_{m,n}$ the rooted tree which has $3$ vertices $\{v_0,v_1,v_2\}$ on the first level, such that $v_0$ has $m$ children and $v_1$ and $v_2$ each have $n$ children.  It is shown that the families of graphs given by $S_{m,2}$ ($m \in \mathbb{N}$) and $B_{m,n}$ ($m,n \in \mathbb{N}$) admit near $\alpha$-valuations, but do not admit $\alpha$-valuations as they satisfy the conditions given by Rosa to belong to the set $S(2,4)$ (see Definition \ref{def:S(2,4)}).

\begin{figure}
    \centering
    \begin{tikzpicture}[baseline=4mm]
    \begin{scope}[radius=1mm]

    \fill(1.5,0.5) circle;
    \fill(1.5,-1) circle;
    \fill(0.5,-2.5) circle;
    \fill(2.5,-2.5) circle;
    \fill(-2.5,-2.5) circle; 
    \fill(-0.5,-2.5) circle;
    \fill(5.5,-2.5) circle; 
    \fill(3.5,-2.5) circle;
    \fill(-1.5,-1) circle; 
    \fill(4.5,-1) circle; 
    \end{scope}
    \begin{scope}[thick]
    
    \draw (-1.5,-1)--(-2.5,-2.5);
    \draw (-1.5,-1)--(-0.5,-2.5);
    \draw (-1.5,-1)--(1.5,0.5);
    
    \draw (1.5,-1)--(2.5,-2.5);
    \draw (1.5,-1)--(0.5,-2.5);
    \draw (1.5,-1)--(1.5,0.5);
    
    \draw (4.5,-1)--(3.5,-2.5);
    \draw (4.5,-1)--(5.5,-2.5);
    \draw (4.5,-1)--(1.5,0.5);

    \node at (1.5,0.75) {$x_0$};
    \node at (-1.5,-2.75) {$\ldots$};
    \node at (-2,-0.5) {$y_0$};

    \node at (4.5,-2.75) {$\ldots$};
    \node at (4.5,-0.5) {$y_2$};
    \node at (1.5,-2.75) {$\ldots$};
    \node at (1,-0.5) {$y_1$};
    \node at (-2.5,-2.75) {$z_{0,0}$};
    \node at (-0.5,-2.75) {$z_{0,m}$};

    \node at (2.5,-2.75) {$z_{1,n}$};
    \node at (0.5,-2.75) {$z_{1,0}$};

    \node at (3.5,-2.75) {$z_{2,0}$};
    \node at (5.5,-2.75) {$z_{2,n}$};

    \end{scope}
    \end{tikzpicture}
    \caption{$B_{m,n}$}
    \label{fig:B_{m,n}}
\end{figure}

We next present a new construction for trees possessing near $\alpha$-valuations, but which do not possess $\alpha$-valuations.  Our approach uses cyclotomy.

\begin{construction}\label{constr:near alpha}
Let $p \geq 11$ be a prime such that $p \not \equiv \pm 1 \mod 8 $.  Write the elements of $\mathrm{GF}(p)$ as $\{0,1,\ldots,p-1\}$ with the natural ordering $0<1< \cdots < p-1$; we will identify these with the corresponding integers $\{0,1,\ldots,p-1\}$ as the labels for the near $\alpha$-valuation.  Let $S$ denote the set of nonzero squares in $\mathrm{GF}(p)$ and $N$ the set of nonsquares; again, we will identify these with the corresponding integers.

We construct a vertex-labelled tree as follows; we define it as a rooted tree with levels $0$, $1$ and $2$:
\begin{itemize}
\item[(i)] In Level $0$, place a single vertex labelled with $0$;
\item[(ii)] in Level $1$, as child nodes of $0$, place vertices labelled with all elements of $N$, and with those elements of $S$ which are greater than $\frac{p}{2}$;
\item[(iii)] in Level $2$, place vertices labelled with each element of $S$ which is less than $\frac{p}{2}$, such that each vertex labelled $s \in S$ occurs as the child node of the vertex labelled with $2s \in N$ (which lies on Level $1$).
\end{itemize}
\end{construction}

\begin{theorem}
Let $p \geq 11$ be a prime such that $p \not \equiv \pm 1 \mod 8 $.  
Then Construction \ref{constr:near alpha} produces a tree possessing a near $\alpha$-valuation but no $\alpha$-valuation. If $p \equiv 5 \mod 8$, this is a rooted tree with $3(p-1)/4$ children, of which $(p-1)/4$ have one child each.
\end{theorem}

\begin{proof}
The cyclotomic classes of order $2$ in $GF(p)$ are $C_0^2=S$ (the set of non-zero squares) and its multiplicative coset $C_1^2=N$ (the set of non-squares). By \cite{Storer}, as $p \not \equiv \pm 1 \mod 8 $, the element $2$ is a nonsquare in $GF(p)$, i.e. $2 \in N$. So $N=2C_0^2=2S$, and so $GF(p)$ is the disjoint union $\{0\} \cup S \cup 2S$.  When $p \equiv 1 \mod 4$, $-1 \in S$ so $x \in S$ if and only if $-x \in S$, i.e. half of the elements of $S$ are greater than $p/2$, namely $(p-1)/4$ elements; in this case Level 1 has $3(p-1)/4$ elements.  

We first show we have a $\beta$-valuation.  Since there are $p$ vertices, labelled with distinct elements of $\{0,1,\ldots, p-1\}$, this is an injective function $b$ from $V(G)$ to $\{0,\ldots,p-1\}$.  With the given vertex labelling, each edge $\{u,v\}$ with $u$ in Level 0 and $v$ in Level 1 is labelled by $|b(v)-b(u)|$ which is $b(v)$ since $b(u)=0$. This yields, as edge-labels, all elements in N and those elements of $S$ which are greater than $\frac{p}{2}$. Now, each edge $\{u,v\}$ with $u$ in Level 1 and $v$ in Level 2 has $b(u)=2s$ and $b(v)=s$ for some $s<p/2$, i.e. $0<s<2s<p$ in the element ordering. Its edge-label is given by $|2s - s| = s$. This yields the remaining elements in $S$ as edge-labels, so we have a $\beta$-valuation. 

By construction, each vertex-label in Level 1 is larger than $0$. Each vertex labelled $s$ in Level 2 is adjacent to the vertex labelled $2s$ in Level 1, and as each $s < p/2$ we have $2s > s$. Hence each vertex is larger than all of its neighbours or smaller than all of its neighbours, so this is a near $\alpha$-valuation. 

Finally, we show that this (unlabelled) tree does not possess any $\alpha$-valuation since it is a member of Rosa's class of counterexamples $S(2,4)$, defined in Definition \ref{def:S(2,4)}.  By construction, the diameter is $4$ if there are at least $2$ edges between Level 1 and Level 2; however to ensure that $T'$ is not a path, we require at least $3$ such edges (the third condition is automatically satisfied).  We claim that there are always at least three $x \in S$ such that $x<p/2$.  When $p \equiv 1 \mod 4$, from above there are $(p-1)/4$ primes less than $p/2$, which is at least $3$ as $p \geq 13$.  When $p \equiv 3 \mod 4$, for $p>32$ we have that $\{1,4,16\} \in S$ (and $\{2,8,32\} \in N$) so this is sufficient.  The only primes less than $32$, for which $2$ is a non-square and which are congruent to $3 \mod 4$ are $\{11,19\}$: direct checking shows that there are at least $3$ non-zero squares less than $p/2$ in each case.
\end{proof}

\begin{example}\label{nearalphaexmp}
Let $p=11$.  Here $S=\{1,3,4,5,9\}$ and $N=\{2,6,7,8,10\}$. Take vertex-labels as follows: Level $0=\{0\}$, Level $1=\{2,6,7,8,9,10\}$ and Level $2=\{1,3,4,5\}$ with edges $\{(0,x): x \in N \cup \{9\}\}$ and $\{(1,2), (3,6), (4,8), (5,10)\}$. This is shown in Figure \ref{fig:nearalphaexmp}.
\end{example}

\begin{figure}
    \centering
    \begin{tikzpicture}[baseline=4mm]
    \begin{scope}[radius=1mm]

    \fill(1.5,0.5) circle;
    \fill(1.5,-1) circle;
    \fill(0,-2.5) circle;
    \fill(3,-2.5) circle; 
    \fill(-1.5,-2.5) circle; 
    \fill(6,-1) circle; 
    \fill(6,-2.5) circle; 
    \fill(0,-1) circle;
    \fill(3,-1) circle; 
    \fill(-1.5,-1) circle; 
    \fill(4.5,-1) circle; 
    \end{scope}
    \begin{scope}[thick]
    \draw (0,-1)--(0,-2.5);
    \draw (0,-1)--(1.5,0.5);
    \draw (3,-1)--(3,-2.5);
    \draw (3,-1)--(1.5,0.5);
    \draw (1.5,-1)--(1.5,0.5);
    \draw (6,-1)--(1.5,0.5);
    \draw (6,-1)--(6,-2.5);
    \draw (-1.5,-1)--(-1.5,-2.5);
    \draw (-1.5,-1)--(1.5,0.5);
    \draw (4.5,-1)--(1.5,0.5);
    \node at (1.5,1) {0};
    \node at (-1.75,-2.75) {1};
    \node at (-1.75,-1.25) {2};
    \node at (-0.25,-2.75) {3};
    \node at (-0.25,-1.25) {6};
    \node at (3.5,-2.75) {4};
    \node at (3.5,-1.25) {8};
    \node at (6.5,-1.25) {10};
    \node at (6.5,-2.75) {5};

    \node at (5,-1.25) {9};
    \node at (1.75,-1.25) {7};
    \end{scope}
    \end{tikzpicture}
    \caption{Example \ref{nearalphaexmp}}
    \label{fig:nearalphaexmp}
\end{figure}

Finally, we consider combining graphs with near $\alpha$-valuations to make new graphs.  In \cite{Snevily}, Snevily defines the weak tensor product of two bipartite graphs:

\begin{definition}
The weak tensor product of two bipartite graphs $G$ and $H$ with partitions $V_1$ and $V_2$ and $W_1$ and $W_2$ respectively is denoted $G \overline{\otimes} H$. It is a graph on the vertices $(V_1 \times W_1) \cup (V_2 \times W_2)$ where vertices $(x,y)$ and $(u,v)$ are adjacent if and only if:
\[\{x,u\} \in E(G) \mbox{ and }\{y,v\} \in E(H).\]
\end{definition}

In \cite{El-Zanati}, it is shown that this can be used to create more graphs with near $\alpha$-valuations.

\begin{theorem}\label{nearalphaprod}
If $G$ and $H$ admit near $\alpha$-valuations $\gamma$ and $\delta$ respectively, then their weak tensor product (with respect to the vertex bipartitions induced by $\gamma$  and $\delta$) admits a near $\alpha$-valuation.
\end{theorem}

We describe the method used in \cite{El-Zanati} to create the near $\alpha$-valuations for $G \overline{\otimes} H$. Since $G$ and $H$ have near $\alpha$-valuations ($\gamma$ and $\delta$ respectively), they are bipartite; we denote the bipartitions corresponding to $V_{small}$ and $V_{large}$ for these near $\alpha$-valuations as $V(G)=V(G)_{small} \cup V(G)_{large} $ and $V(H)=V(H)_{small} \cup V(H)_{large}$. (Note that it is essential to use this specific choice of bipartitions.)  We also suppose $G$ and $H$ have $m$ and $n$ edges respectively. The function $\sigma$ defined by:
\begin{itemize}
\item $\sigma(v_1,w_1) = m\delta(w_1)+\gamma(v_1)$ for $(v_1,w_1) \in V(G)_{small} \times V(H)_{small}$
\item $\sigma(v_2,w_2) = m(\delta(w_2)-1)+\gamma(v_2)$ for $(v_2,w_2) \in V(G)_{large} \times V(H)_{large}$
\end{itemize}
is a near $\alpha$-valuation for $G \overline{\otimes} H$.

\section{Digraph-defined EDFs from near $\alpha$-valuations}

We now present theorems for constructing digraph-defined EDFs in cyclic groups from any graphs possessing near $\alpha$-valuations.  Various results about CEDFs from $\alpha$-valuations which were explicitly or implicitly present in \cite{PatStiCEDF} are special cases of these. For completeness and to illuminate how the labelling properties are used, we give complete proofs.

\begin{lemma}\label{directtopm}
Let $G^*$ be a directed graph, and let $G$ be the corresponding undirected graph.  Let $p:V(G) \rightarrow \{0,1,\ldots,n\}$ ($n=|E(G)|-1$) be a one-to-one map. 
If the multiset equation 
    $$\bigcup_{(u,v)\in \overrightarrow{E}(G^*)} \{ p(v)-p(u) \} =\{1, \ldots n\}$$ 
    holds, and for each directed edge $(u,v)\in \overrightarrow{E}(G^*)$ we have that $p(u) < p(v)$, 
    then $G$ has a near $\alpha$-valuation given by $p$.
\end{lemma}

We now establish (via a series of results) that the balanced blow-up of a graph by a factor of $l$ has a near $\alpha$-valuation if the original does.
\begin{theorem}\label{pblowlarge}
    Let $G$ be a graph with a near $\alpha$-valuation $p$, and let $l$ be a positive integer. The blow-up $G^p_{1,l}$ has an near $\alpha$-valuation, $p'$.
\end{theorem}
\begin{proof}
Consider the natural orientation $G^p$ of $G$.  By Lemma \ref{betatodirect}, the multiset equation $\bigcup_{(x,y)\in \overrightarrow{E}(G^p)} \{(p(y)-p(x) )\} =\{1, \ldots n\}$ holds. Consider a directed edge $(x,y) \in \overrightarrow{E}(G^p)$; then $p(y)-p(x)=k$ for some $k\in \{1,\ldots,n\}$, due to the natural orientation $y \in V_{large}$ and $x \in V_{small}$. To form $G^p_{1,l}$:
\begin{itemize}
\item We replace $y$ with the $l$-set of vertices labelled by $V_y=\bigcup_{i=0}^{l-1}\{lp(y)-i\}$.
Note that $p(y) \neq 0$ as $y \in V_{large}$, so $p(y) \geq 1$ and so $lp(y)-i \geq 1$ for all $0 \leq i \leq l-1$.
\item We replace $x$ with a singleton set of vertices labelled by $V_x = \{lp(x)\}$.
\item The directed edge $(x,y)$ is replaced with the $l$ directed edges from $V_x$ to $V_y$.  The corresponding differences for these edges are:
\[lp(y)-i-lp(x)= l(p(y)-p(x))-i\]
where $0 \leq i \leq l-1$.  It is clear that all edges in the blow-up are formed in this way. 
\end{itemize}

We claim that the given vertex labelling defines a near $\alpha$-valuation $p'$ on $G^p_{1,l}$.

Taking the multiset union over all edges in $G^p$ yields the multiset of differences corresponding to the set of all edges in $G^{p'}_{1,l}$, namely 
\[\bigcup_{(u,v) \in G^{p'}_{1,l}}\{(p'(v)-p'(u))\}=\bigcup_{(x,y) \in G^p}\bigcup_{i =0}^{l-1}\{l(p(y)-p(x))-i\}=\bigcup_{i =0}^{l-1}\bigcup_{k=1}^{n}\{lk-i\}\]
\[=\bigcup_{k=1}^{n}[l(k-1)+1,lk] = [1,nl].\]

Now note that, since $p$ is a near $\alpha$-valuation, for each $ y \in V_{large}$ and $x \in V_{small}$ we have $p(y) > p(x)$. Hence in $G^p_{l,1}$ for any $s \in V_y$ and $t \in V_x$:
\[ p'(s) \geq l p(y) > l p(x) = p'(t).\]
Hence by Lemma \ref{directtopm} this is a near $\alpha$-valuation.
\end{proof}

\begin{theorem}\label{pblowsmall}
    Let $G$ be a graph with a near $\alpha$-valuation $p$, and let $l$ be a positive integer. The blow-up $G^p_{l,1}$ has an near $\alpha$-valuation.
\end{theorem}
\begin{proof}
Consider the natural orientation $G^p$: from Lemma \ref{betatodirect} we know $\bigcup_{(x,y)\in \overrightarrow{E}(G^p)} (p(y)-p(x) ) =\{1, \ldots n\}$. Consider a directed edge $(x,y) \in \overrightarrow{E}(G^p)$ then $p(y)-p(x)=k$ where $k\in \{1,\ldots,n\}$ by the natural orientation $y \in V_{large}$ and $x \in V_{small}$. 

\begin{itemize}
\item We replace $y$ with a singleton set of vertices labelled by $V_y=\{lp(y)\}$.
\item We replace $x$ with the $l$ set of vertices labelled by $V_x =\bigcup_{i=0}^{l-1}\{lp(x)+i\}$.
Since $x \in V_{small}$, $x \neq n$, i.e. $x \leq n-1$ and so $lp(x)+i \leq ln-1$ for $0 \leq i \leq n-1$.
\item  The directed edge $(x,y)$ is replaced with the $l$ directed edges from $V_x$ to $V_y$; these are labelled as:
\[lp(y)-(lp(x)+i)= l(p(y)-p(x))-i\]
where $0 \leq i \leq l-1.$ All edges in the blow-up are formed in this way. 
\end{itemize}

Taking the multiset union over all edges in $G^a$ will give the multiset of differences corresponding to the set of all edges in $G^p_{l,1}$, namely 
\[\bigcup_{(u,v) \in G^{p'}_{l,1}}\{(p'(v)-p'(u))\}=\bigcup_{(x,y) \in G^p}\bigcup_{i =0}^{l-1}\{l(p(y)-p(x))-i\}=\bigcup_{i =0}^{l-1}\bigcup_{k=1}^{n}\{lk-i\}\]
\[=\bigcup_{k=1}^{n}[l(k-1)+1,lk] = [1,nl].\]
Note that, since $p$ is a near $\alpha$-valuation, for each $ y \in V_{large}$ and $x \in V_{small}$ we have $p(y) >  p(x)$. Hence in $G^p_{l,1}$, for any $s \in V_y$ and $t \in V_x$:
    \[a'(s) = a(y)l >  a(x)l \geq a'(t)\]
Hence by Lemma \ref{directtopm} this is a near $\alpha$-valuation.
\end{proof}

This leads to the following theorem.

\begin{theorem}\label{blowp}
If a graph $G$ has an near $\alpha$-valuation, then its lexicographic product with $K^c_l$ has an near $\alpha$-valuation.
\end{theorem}
\begin{proof}
    This follows from Theorem \ref{pblowlarge}, Theorem \ref{pblowsmall} and Corollary \ref{lexiblow}.
\end{proof}

We prove the following which will be important in applying the blow-up construction.

\begin{lemma}\label{blowdirec}
Let $A_i$ be the set of vertex labels in $G \cdot K^c_l$ corresponding to the vertex labelled $i$ in $G$. If $i<j$ then all elements of $A_i$ are smaller than all elements of $A_j$.
\end{lemma}
\begin{proof}
From Theorem \ref{pblowlarge} and Theorem \ref{pblowsmall}, for any two distinct vertices $i, j \in G$ where $i \leq  j-1$, there are 3 cases:
\begin{itemize}
\item[(1)] $i,j \in V_{small}$;
\item[(2)] $i,j \in V_{large}$;
\item[(3)] $i \in V_{small} \text{ and } j \in V_{large}$
\end{itemize}
For the first two cases, we need only consider the blow-up of the part of the bipartition which contains both $i$ and $j$, since blowing up the other part only multiplies both labels in a way which preserves the order.  In Case 1, the maximum element in $A_i$ is given by $il+l-1$ and the minimum element in $A_j$ is given by $jl$; the result follows since $il+l-1 \leq (j-1)l+l-1 = jl-1 < jl$.  In Case 2, the maximum element in $A_i$ is $il$ and the minimum element in $A_j$ is $jl-(l-1)=j(l-1)+1$, so the result follows as $il \leq (j-1)l < (j-1)l+1$.  In Case 3 we need to consider the full blow-up of both parts.  We will consider first blowing up $V_{small}$, then $V_{large}$. After blowing up $V_{small}$, the maximum label in $A_i$ is $il+l-1$, and the label of $j$ has become $jl$.  Next, after blowing up $V_{large}$ the maximum label in $A_i$ is $l(il+l-1)=il^2+l^2-l$, and the minimum label of $A_j$ is $jl^2-(l-1)$.  We can then observe that: $il^2+l^2-l \leq (j-1)l^2+l^2-1 \leq jl^2-l < jl^2 -(l-1)$.  So all elements of $A_i$ are smaller than all elements of $A_j.$
\end{proof}

\begin{remark}\label{cordir}
Lemma \ref{blowdirec} shows that the natural orientation of any edge between vertices in $G$ and any edge between vertices in the corresponding sets in $G \cdot K_{l}^C$ is the same. 
\end{remark}

\begin{theorem}\label{thm:nearalpha-blowup}
Let a graph $G=(V,E)$ have an near $\alpha$-valuation $p$. Then there exists an $(|E|l^2 + 1, |V|, l, 1; G^p)$-EDF in $\mathbb{Z}_{|E|l^2+1}$.
\end{theorem}
\begin{proof}
To obtain the EDF, we consider the graph $H=G \cdot K_{l}^C$ (viewed as a blow-up of $G$).
As $G$ has a near $\alpha$-valuation, by Theorem \ref{blowp}, so does $H$.  Call this near $\alpha$-valuation $p'$, where by Remark \ref{cordir} we know that the orientation of any edge between a pair of vertices $u$ and $v$ of $G$, and any edge in $H$ between vertices in $V_u$ and $V_v$, is the same. We have 
$$\bigcup_{(x,y)\in \overrightarrow{E}(H^{p'})} \{p'(y)-p'(x)\} =\{1, \ldots nl^2\}.$$ 
Recall from the proof of Theorem \ref{pblowlarge} and Theorem \ref{pblowsmall} that each edge $\{u,v\} \in E(G)$ corresponds to the $l^2$ edges from $V_u$ to $V_v$ in $H$, and this accounts for all edges of $H$. 

For each $u \in G$, define the $l$-subset $A_u$ of $\mathbb{Z}_{nl^2+1}$ by $\{p'(x): x \in V_u\}$. Then:
    \[\bigcup_{(u,v)\in \overrightarrow{E}(G^p)} \Delta(A_v,A_u ) = \bigcup_{(x,y)\in \overrightarrow{E}(H^{p'})} \{p'(y)-p'(x)\}=\{1, \ldots nl^2\}=\mathbb{Z}_{nl^2+1}\backslash\{0\}.\]
\end{proof}

\begin{example}
Consider the following graph $G$ with $5$ vertices $\{v_1,v_2,v_3,v_4,v_5\}$ and $5$ edges $\{\{v_1,v_2\}, \{v_2,v_3\}, \{v_3,v_4\}, \{v_1,v_4\}, \{v_1,v_5\}$. It has a near $\alpha$-valuation $p$ - which is in fact an $\alpha$-valuation - given by $p(v_1)=0,p(v_2)=3, p(v_3)=2, p(v_4)=4$ and $p(v_5)=5$, as shown below.  
\begin{center}  
    \begin{tikzpicture}[baseline=4mm, scale=0.6]
    \begin{scope}[radius=1mm]
    \fill(0,2) circle;
    \fill(0,-1) circle; 
    \fill(6,2) circle; 
    \fill(3,-1) circle;
    \fill(3,2) circle;
    \node at (1.5,2.5) {4};
    \node at (-0.5,2.5) {\textbf{4}};
    \node at (-0.5,-1.5) {\textbf{2}};
    \node at (1.5,-1.5) {1};
    \node at (-0.5,0.5) {2};
    \node at (3,2.5) {\textbf{0}};
    \node at (6,2.5) {\textbf{5}};
    \node at (4.5,2.5) {5};
    \node at (3.5,0.5) {3};
    \node at (3.5,-1.5) {\textbf{3}};
    \end{scope}
    \begin{scope}[thick]
    \path [draw=black,postaction={on each segment={mid arrow=black}}](3,2)--(6,2);
    \path [draw=black,postaction={on each segment={mid arrow=black}}](3,2)--(0,2);
    \path [draw=black,postaction={on each segment={mid arrow=black}}](3,2)--(3,-1);
    \path [draw=black,postaction={on each segment={mid arrow=black}}](0,-1)--(3,-1);
    \path [draw=black,postaction={on each segment={mid arrow=black}}](0,-1)--(0,2);
    \end{scope}
    \end{tikzpicture}
\end{center}
Using Theorems \ref{pblowlarge} and \ref{pblowsmall}, we can blow this up with $l=3$ to give the sets:
\[V_{v_1}=\{0,1,2\}, V_{v_2}=\{21,24,27\}, V_{v_3}=\{18,19,20\},V_{v_4}=\{30,33,36\},V_{v_5}=\{39,42,45\}.\]
Here $|E|=5$, $l=3$ and so $|E|l^2+1=46$.  Viewing these as subsets of the group  $\mathbb{Z}_{46}$, the differences between the elements in the sets of the blow-up, with orientation given by the natural orientation of $ a$, yield all non-zero group elements precisely once: $\Delta(V_{v_2},V_{v_1})=[19,27],\Delta(V_{v_4},V_{v_1})=[28,36],
\Delta(V_{v_5},V_{v_1})=[37,45],\Delta(V_{v_2},V_{v_3})=[1,9]$ and $\Delta(V_{v_4},V_{v_3})=[10,18]$.
Hence these sets form a $(46,5,3,1,G^a)$-EDF in $\mathbb{Z}_{46}$.
\end{example}

\begin{remark}
By Remark \ref{20nearalphatree}, for all trees $T$ with at most 20 vertices there exists a $(|E|l^2 + 1, |V|, l, 1; T^p)$-EDF in $\mathbb{Z}_{|E|l^2+1}$. Conjecture \ref{nearalphatree}, if true, would guarantee that for all trees $T$ there exists a $(|E|l^2 + 1, |V|, l, 1; T^p)$-EDF in $\mathbb{Z}_{|E|l^2+1}$.
\end{remark}

\section{Oriented $\beta$- and near $\alpha$-valuations}

We have considered graphs that possess $\beta$-valuations and near $\alpha$-valuations. However, in cases where no such valuation exists, it can be possible to obtain a vertex-labelling such that, when the edges are suitably directed, the resulting edge-labels gives each non-zero element exactly once (working with modular arithmetic). Directed graphs with such labellings are studied by Bloom and Hsu in \cite{BloHsu} (where they are called \emph{graceful digraphs}).  We introduce the following notation for the valuation itself.  All of our digraphs will be oriented graphs. We identify $\mathbb{Z}_{n+1}$ with $\{0,\ldots,n\}$
in the natural way, with ordering $0<1< \cdots<n$.

\begin{definition}
Let $G$ be an digraph with $n$ edges and let $b$ be a one-to-one mapping from $V(G)$ to $\{0,1,\ldots,n\}$. We will call $b$ of $G$ an \emph{oriented $\beta$-valuation} if the following multiset equation holds:
    \[\bigcup_{(u,v) \in \overrightarrow{E}(G)} \{ b(v)-b(u) \mod n+1\} = \mathbb{Z}_{n+1}\backslash\{0\}\]
\end{definition}

\begin{example}\label{orbeta}
    Consider the digraph $G$ with the labelling given below.  The digraph has $7$ edges; working modulo $8$, we obtain an oriented $\beta$-valuation.
    
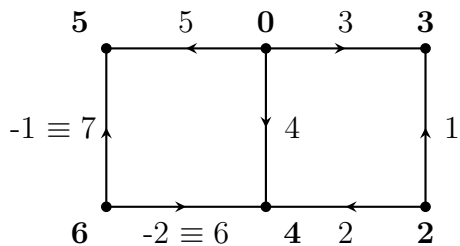
\begin{center}
    \begin{tikzpicture}[baseline=4mm,scale=0.7]
    \begin{scope}[radius=1mm]
    \fill(0,2) circle;
    \fill(0,-1) circle; 
    \fill(6,-1) circle;
    \fill(6,2) circle; 
    \fill(3,-1) circle;
    \fill(3,2) circle;
    \node at (1.5,2.5) {5};
    \node at (4.5,2.5) {3};
    \node at (4.5,-1.5) {2};
    \node at (-0.5,2.5) {\textbf{5}};
    \node at (-0.5,-1.5) {\textbf{6}};
    \node at (1.5,-1.5) {-2 $\equiv 6$};
    \node at (-1,0.5) {-1 $\equiv 7$};
    \node at (3,2.5) {\textbf{0}};
    \node at (6,2.5) {\textbf{3}};
    \node at (6,-1.5) {\textbf{2}};
    \node at (3.5,0.5) {4};
    \node at (6.5,0.5) {1};
    \node at (3.5,-1.5) {\textbf{4}};
    \end{scope}
    \begin{scope}[thick]
    \path [draw=black,postaction={on each segment={mid arrow=black}}](3,2)--(6,2);
    \path [draw=black,postaction={on each segment={mid arrow=black}}](6,-1)--(3,-1);
    \path [draw=black,postaction={on each segment={mid arrow=black}}](6,-1)--(6,2);
    \path [draw=black,postaction={on each segment={mid arrow=black}}](3,2)--(0,2);
    \path [draw=black,postaction={on each segment={mid arrow=black}}](3,2)--(3,-1);  
    \path [draw=black,postaction={on each segment={mid arrow=black}}](0,-1)--(3,-1);
    \path [draw=black,postaction={on each segment={mid arrow=black}}] (0,-1)--(0,2);
    \end{scope}
    \end{tikzpicture}
    \end{center}
    Note that this vertex-labelling is not a $\beta$-valuation for the corresponding undirected graph, as $1$ and $2$ both would appear twice and $6$ and $7$ would not appear at all.
\end{example}

In \cite{BloHsu} Bloom and Hsu note that, if a graph $G$ has a $\beta$-valuation $b$, then there exists an orientation of $G$ which has an oriented $\beta$-valuation, namely $G^{b}$. However, if $G$ is an digraph with an oriented $\beta$-valuation, then it is not guaranteed that the corresponding undirected graph will have a $\beta$-valuation. 

Oriented versions can be useful when no appropriate labelling of the underlying graph exists. For example, we have seen that no $\beta$-valuation exists for the undirected graph $C_m$ when $m \equiv 2 \mod 4$.

\begin{proposition}\label{thm:mcycle2mod4}
Let $m \equiv 2 \mod 4$. Consider the graph $C_m$ with $V(C_m)=\{v_1,\ldots,v_m\}$ and $E(C_m)=\{\{v_1,v_2\}, \ldots, \{v_m,v_1\}\}$. Then the vertex-labelling
\[a(v_i) = \begin{cases} 
          \frac{i-1}{2} & i \textit{ odd} \\
          m +1- \frac{i}{2} & i \textit{ even, }i \leq \frac{m}{2} \\
          m - \frac{i}{2} & i \textit{ even, }i > \frac{m}{2} \\
       \end{cases}
\]
is an oriented $\beta$-valuation for the digraph $C_m$ with orientation corresponding to the natural orientation from $a$, but with the orientation of $(v_m,v_1)$ reversed.
\end{proposition}
\begin{proof}
Consider the edge-labels arising from $a$ on $C_m$:
\begin{itemize}
\item[(1)] $\bigcup_{\substack{1 \leq i \leq \frac{m}{2}-2 \\i \text{ odd }}}\{a(v_{i+1})-a(v_{i}) \} = \bigcup_{\substack{1 \leq i \leq \frac{m}{2}-2 \\i \text{ odd }}}\{m+1-i\}=\{\frac{m}{2}+3, \frac{m}{2}+5, \ldots, m-2,m\}$.
\item[(2)] $\bigcup_{\substack{3 \leq i \leq \frac{m}{2} \\i \text{ odd }}}\{a(v_{i-1})-a(v_i) \} = \bigcup_{\substack{3 \leq i \leq \frac{m}{2} \\i \text{ odd }}}=\{\frac{m}{2}+2,\frac{m}{2}+4,\ldots, m-3,m-1\}$.
\item[(3)] $\bigcup_{\substack{\frac{m}{2} \leq i \leq m-1 \\i \text{ odd }}}\{a(v_{i+1})-a(v_i)\} = \bigcup_{\substack{\frac{m}{2} \leq i \leq m-1 \\i \text{ odd }}}=\{1,3, \ldots, \frac{m}{2}-2, \frac{m}{2}\}$.
\item[(4)] $\bigcup_{\substack{\frac{m}{2}+2 \leq i \leq m-1 \\i \text{ odd }}}\{a(v_{i-1})-a(v_i) \} = \bigcup_{\substack{\frac{m}{2}+2 \leq i \leq m-1 \\i \text{ odd }}}=\{2,4, \ldots, \frac{m}{2}-3,\frac{m}{2}-1\}$.
\end{itemize}  
Finally we also have: $a(v_m)-a(v_1) = \frac{m}{2}-0 = \frac{m}{2}$.
So we obtain every element in $[1,m]$ exactly once except for $\frac{m}{2}$, which we obtain twice, and $\frac{m}{2}+1$, which we do not obtain. These are precisely the edge-labels arising from the oriented version of this vertex-labelling on the digraph $C_m$ with the natural orientation given by $a$.

Now consider the new orientation obtained from the natural one by reversing the direction of $(v_m,v_1)$.  The vertex labels are unchanged, and the edge labels are unchanged with the exception of the final label, which becomes:
$a(v_1)-a(v_m) = 0-\frac{m}{2} \equiv \frac{m}{2}+1 (\mod m+1)$ . This yields every element in $\mathbb{Z}_{m+1}\backslash\{0\}$ as an edge-label exactly once, as required.
\end{proof}
Note that this vertex-labelling is the same one used in Theorem \ref{cyc} to obtain an $\alpha$-valuation for $C_m$ with $m \equiv 0 \mod 4$.  It was also used in the construction of an $m$-set CEDF with $m \equiv 2 \mod 4$ in \cite{PatStiCEDF}, but here the orientation of the cycle is not the same.

We can generalise the idea from the above proof to any digraph with the necessary conditions. For a set $A$, we define the set $-A=\{-a: a\in A\}$. 

\begin{lemma}\label{selfflip}
    Let $G$ be a digraph and $b$ a vertex labelling of $G$.  If the multiset equation
    \[\bigcup_{(u,v) \in \overrightarrow{E}(G)} \{ b(v)-b(u) \mod n+1\} = \mathbb{Z}_{n+1}\backslash(\{0\}\cup -B) \cup B\] 
    is satisfied for some set $B \subset \mathbb{Z}_{n+1}\backslash\{0\}$, then there exists a digraph $G'$ such that $b$ is an oriented $\beta$-valuation of $G'$. 
\end{lemma}
\begin{proof}
For each $b \in B$, reversing the orientation of one of the two edges with edge-label $b$ yields the edge-label $-b$.
\end{proof}

The following is a new construction for an oriented $\beta$-valuation using cyclotomy.
\begin{proposition}\label{prop:orientedcyclotomic}
Let $p$ be prime and consider the rooted tree $S_{(p-1)/2,2}$.  Define its vertex-set as follows: the root is $r$, $\{u_0,\ldots,u_{(p-1)/2-1}\}$ are the set of children of $r$, and $v_i$ is the child of $u_i$ ($0 \leq i \leq (p-1)/2-1\}$.   Let $\alpha$ be a primitive element of $GF(p)$.  Denote by $S$ the set of nonzero squares and by $N$ the set of nonsquares in $GF(p)$. Then
\begin{itemize}
\item[(i)] if $\alpha-1 \in S$, the vertex-labelling $b$ given by: $b(r)=0$, $b(u_i)=\alpha^{2i+1}$ and $b(v_i)=\alpha^{2i}$ is an oriented $\beta$-valuation for the oriented $S_{(p-1)/2,2}$ with directed edges $\{(r,u_i): 0 \leq i \leq (p-1)/2-1\} \cup \{(v_i,u_i): 0 \leq i \leq (p-1)/2\}$.
\item[(ii)] if $\alpha-1 \in N$, the vertex-labelling $b$ given by: $b(r)=0$, $b(u_i)=\alpha^{2i}$ and $b(v_i)=\alpha^{2i+1}$ is an oriented $\beta$-valuation for the oriented $S_{(p-1)/2,2}$ with directed edges $\{(r,u_i): 0 \leq i \leq (p-1)/2-1\} \cup \{(u_i,v_i): 0 \leq i \leq (p-1)/2\}$.
\end{itemize}
\end{proposition}

\begin{remark}
If no modular arithmetic is used in the edge-label differences in case (i) of Proposition \ref{prop:orientedcyclotomic}, then this is a (standard) near $\alpha$-valuation for the graph $S_{(p-1)/2,2}$; for $p=7$ and $\alpha=3$, this is the near $\alpha$-valuation given Theorem \ref{thm:nearprops} for the Rosa star.
\end{remark}

In \cite{BloHsu}, digraphs with known oriented $\beta$-valuations are used to obtain new digraphs and oriented $\beta$-valuations; we briefly state some results here.

\begin{lemma}\label{flip}
    Let $G$ be a digraph with an oriented $\beta$-valuation $b$. Let $G'$ be the digraph given by reversing the orientation of edges labelled by $e$ and $-e \mod n+1$. Then $b$ is also an oriented $\beta$-valuation for $G'$
\end{lemma}

Applying this to each edge in a digraph $G$ creates a family $\mathcal{F}$ of up to $2^{\lfloor \frac{m}{2} \rfloor}$ distinct digraphs with the same vertex set as $G$. The following is a consequence of Lemma \ref{flip}. 

\begin{corollary}
    Let $G$ be a digraph with $n$ edges, with an oriented $\beta$-valuation $b$. Then $b$ is an oriented $\beta$-valuation for all digraphs in $\mathcal{F}$.
\end{corollary}

Other ways to obtain a new oriented $\beta$-valuation from an existing one are given below.
\begin{lemma}\label{addorbeta}
Let $G$ be a digraph with $n$ edges and an oriented $\beta$-valuation $b$. Working modulo $n+1$:
\begin{itemize}
\item[(i)] if $\operatorname{gcd}(n+1,k)=1$, then $b'(x)=k \cdot b(x)$ is also an oriented $\beta$-valuation of $G$;
\item[(ii)] for any $m \in \mathbb{Z}$, $b'(x)=b(x)+m$ is also an oriented $\beta$-valuation of $G$.
\end{itemize}
\end{lemma}
\begin{proof}
We prove (i). As $b$ is an oriented $\beta$-valuation $b$ for $G$, $\bigcup_{(u,v) \in G}\{b(v)-b(u) \mod n+1\} = \mathbb{Z}_{n+1}\backslash\{0\}$ and as $\operatorname{gcd}(n+1,k)=1$, $\{kx:x \in \mathbb{Z}_{n+1} \backslash\{0\} \} = \mathbb{Z}_{n+1}\backslash\{0\}$.  Hence
$\bigcup_{(u,v) \in G}\{b'(v)-b'(u) \mod n+1\}\equiv \bigcup_{(u,v) \in G}\{k(b(v)-b(u)) \mod n+1\} = \{kx:x \in \mathbb{Z}_{n+1} \backslash\{0\} \}  = \mathbb{Z}_{n+1}\backslash\{0\}$.
\end{proof}

\begin{example}
Applying Lemma \ref{addorbeta}(ii) to Example \ref{orbeta} with $m=2$ gives an oriented near $\alpha$-valuation (which is in fact an $\alpha$-valuation for the undirected graph, with $x=4$).
\end{example}

As with $\beta$-valuations, we can use oriented $\beta$-valuations to form digraph-defined EDFs.

\begin{theorem}\label{orbetaEDF}
Let $G$ be a digraph with $n$ edges, with an oriented $\beta$-valuation $b$. Then there exists a $(n,m,1,1;G)$-EDF in $\mathbb{Z}_{n+1}$.
\end{theorem}
\begin{proof}
As $b$ is an oriented $\beta$-valuation for $G$, $\bigcup_{(u,v)\in E(G)} b(v)-b(u) \mod n+1 = \{1, \ldots n\}$.
Define the following subsets of $\mathbb{Z}_{n+1}$ (where we identify the elements of $\mathbb{Z}_{n+1}$ with the corresponding integers in $\{0,\ldots,n\}$): let $A_v=\{b(v)\}$ for each $v \in V(G)$. Then we have:
\[\bigcup_{(u,v)\in E(G)} \Delta(A_v,A_u)= \{1, \ldots n\}=\mathbb{Z}_{n+1}\backslash\{0\}.\]
Hence $(A_{v_1},A_{v_2},\ldots,A_{v_{m}})$ forms a $(n,m,1,1;G)$-EDF in $\mathbb{Z}_{n+1}$.
\end{proof}

In order to allow the use of the blow-up approach to make EDFs with larger sets, we next introduce the notion of an oriented $\beta$-valuation with an extra condition.

\begin{definition}An oriented near $\alpha$-valuation $b$ of a digraph $G$ with $n$ edges is an oriented $\beta$‐valuation from $V(G)$ to $\{0,\ldots,n\}$ which satisfies the following extra condition: there is a partition of $V(G)$ into $V_{small}$ and $V_{large}$ such that
\begin{itemize}
\item for each $v \in V_{small}, b(v) < b(u)$ for all vertices $u$ adjacent to $v$ in the underlying graph, and 
\item for each $v \in V_{large}, b(v) > b(u)$ for all vertices $u$ adjacent to $v$ in the underlying graph.
\end{itemize}
\end{definition}

\begin{example}\label{oralp}
    The digraph $H$ below, with $8$ directed edges, exhibits an oriented near $\alpha$-valuation that is not a near $\alpha$-valuation (due to the directed edge with label $-7 \equiv 2 \mod 9$):
    
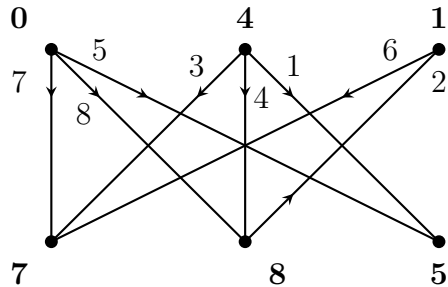
\begin{center}
    \begin{tikzpicture}[baseline=4mm,scale=0.85]
    \begin{scope}[radius=1mm]
    \fill(0,2) circle;
    \fill(0,-1) circle; 
    \fill(6,-1) circle;
    \fill(6,2) circle; 
    \fill(3,-1) circle;
    \fill(3,2) circle;
    \node at (-0.5,2.5) {\textbf{0}};
    \node at (-0.5,-1.5) {\textbf{7}};
    \node at (3,2.5) {\textbf{4}};
    \node at (6,2.5) {\textbf{1}};
    \node at (6,-1.5) {\textbf{5}};
    \node at (3.5,-1.5) {\textbf{8}};
    \node at (2.25,1.75) {3};
    \node at (-0.5,1.5) {7};
    \node at (0.5,1) {8};
    \node at (0.75,2) {5};
    \node at (3.25,1.25) {4};
    \node at (3.75,1.75) {1};
    \node at (5.25,2) {6};
    \node at (6,1.5) {2};
    \end{scope}
    \begin{scope}[thick]
    \path [draw=black,postaction={on each segment={start arrow=black}}](0,2)--(0,-1);
    \path [draw=black,postaction={on each segment={start arrow=black}}](0,2)--(3,-1);
    \path [draw=black,postaction={on each segment={start arrow=black}}](3,2)--(3,-1);
    \path [draw=black,postaction={on each segment={start arrow=black}}](0,2)--(6,-1);
    \path [draw=black,postaction={on each segment={start arrow=black}}](3,2)--(0,-1);  
    \path [draw=black,postaction={on each segment={start arrow=black}}](3,2)--(6,-1);
    \path [draw=black,postaction={on each segment={start arrow=black}}] (6,2)--(0,-1);
    \path [draw=black,postaction={on each segment={start arrow=black}}] (3,-1)--(6,2);
    \end{scope}
    \end{tikzpicture}
    \end{center}
    \end{example}

\begin{example}\label{orcycle}
The oriented $\beta$-valuation for the oriented $C_m$ ($m \equiv 2 \mod 4$) defined in Theorem \ref{thm:mcycle2mod4} is in fact an oriented near $\alpha$-valuation, as it can be checked that the vertex bipartition property is satisfied.
\end{example}

We can extend the blow-up theorems to this type of valuation.  

\begin{theorem}\label{orblowlarge}
    Let $G$ be a digraph with an oriented near $\alpha$-valuation $p$, and let $l$ be a positive integer. The blow-up $G_{1,l}$ has an oriented near $\alpha$-valuation $p'$.
\end{theorem}
\begin{proof}
Consider a directed edge $(u,v) \in \overrightarrow{E}(G)$; as $p$ is an oriented near $\alpha$-valuation, $p(v)-p(u) \equiv k \mod n+1$
for some $k\in \{1,\ldots,n\}$.  By the second condition for an oriented near $\alpha$-valuation, one of the vertices $u,v$ is in $V_{large}$ and the other is in $V_{small}$. For each $k \in \{1,\ldots,n\}$, it is obtained as the label $p(v)-p(u)$ of precisely one directed edge $(u,v)$ of $G$.  Denote by $P$ the set of $k\in \{1,\ldots,n\}$ which arise as positive integer differences (i.e. from $(u,v)$ where $u \in V_{small}$ and $v \in V_{large}$), and by $N$ the set of $k \in \{1,\ldots,n\}$ which arise as negative integer differences (i.e. from $(u,v)$ where $u \in V_{large}$ and $v \in V_{small}$).  We work in the integers until the final step of the proof.  

The first case (i.e. labels from $P$) is exactly the same as in the near $\alpha$ case. 
\begin{itemize}
\item We replace $v$ with the $l$-set of vertices labelled by $V_v=\bigcup_{i=0}^{l-1}\{lp(v)-i\}$.
\item We replace $u$ with a singleton set of vertices labelled by $V_u = \{lp(u)\}$.
\item The directed edge $(u,v)$ is replaced with the $l$ directed edges from $V_u$ to $V_v$; the corresponding differences for these edges are:
$lp(v)-i-lp(u) =  l(p(v)-p(u))-i = lk-i$
where $0 \leq i \leq l-1$ and $k \in P$. 
\end{itemize}

Taking multiset unions of all of these differences:
\[
\bigcup_{\substack{(u,v) \in G \\ u \in V_{small}}} \left( \bigcup_{i=0}^{l-1} \{ (l(p(v)-p(u))-i) \} \right)
= \bigcup_{k\in P} \bigcup_{i=0}^{l-1} \{ lk-i \} = \bigcup_{k \in P}[l(k-1)+1,lk]. 
\]

Now consider the directed edges with labels corresponding to $N$. 
\begin{itemize}
\item We replace $u$ with the $l$-set of vertices labelled by $V_u=\bigcup_{i=0}^{l-1}\{lp(u)-i\}$.
\item We replace $v$ with a singleton set of vertices labelled by $V_v = \{lp(v)\}$.
\item The directed edge $(u,v)$ is replaced with the $l$ directed edges from $V_u$ to $V_v$.  Since $p(v)-p(u)$ is a negative integer, the actual integer difference corresponding to $k \in N$ in the $\mathbb{Z}_{n+1}$ setting was in fact $-(n+1-k)$.  So the corresponding integer differences for these edges in the blow-up setting are:
\[lp(v)-(lp(u)-i)=l(p(v)-p(u))+i = l(-(n+1-k))+i=lk-(l-1)+i-(nl+1).\]
where $0 \leq i \leq l-1$ and $k \in N$.  
\end{itemize}
Taking multiset unions of all of these differences:
\[
\bigcup_{(u,v) \in G, v \in V_{small}} \left( \bigcup_{i=0}^{l-1} \{ (l(p(v)-p(u)-(n-1))+i) \} \right)
= \bigcup_{k\in N} \bigcup_{i=0}^{l-1} \{ lk-(l-1)+i-(nl+1) \}\]
\[= \bigcup_{k \in N}[l(k-1)+1-(nl+1),lk-(nl+1)] \equiv \bigcup_{k \in N}[l(k-1)+1,lk]\mod nl+1. \]

It is clear that all edges in the blow-up are formed by one of the two above cases. Taking the multiset union over $k \in P$ and $k \in N$ gives all $k \in \{1, \ldots ,n \}$, so (working now in $\mathbb{Z}_{nl+1}$):
\[\bigcup_{(x,y) \in G_{1,l}} \{p'(y)-p'(x)\} \equiv \bigcup_{k=1}^n[l(k-1)+1,lk] \mod nl+1 = \mathbb{Z}_{nl+1} \setminus \{0\}.\]

Now note that, since $p$ is an oriented near $\alpha$-valuation, for each $ y \in V_{large}$ and $x \in V_{small}$ we have $p(y) > p(x)$ (in the ordering $0<1< \cdots < n$). Hence in $G_{1,l}$ for any $s \in V_y$ and $t \in V_x$:
    \[p'(s) \geq p(y)l >   p(x)l = p'(t)\]
(in the ordering $0<1< \cdots< nl$).
Hence $p'$ is an oriented near $\alpha$-valuation.
\end{proof} 

\begin{theorem}\label{orblowsmall}
    Let $G$ be a digraph that has an oriented near $\alpha$-valuation $p$, and let $l$ be a positive integer. The blow-up $G_{l,1}$ has an oriented near $\alpha$-valuation.
\end{theorem}
\begin{proof}

Consider a directed edge $(u,v) \in \overrightarrow{E}(G)$; as $p$ is an oriented near $\alpha$-valuation, 
$p(v)-p(u) \equiv k \mod n+1$
for some $k\in \{1,\ldots,n\}$.  By the second condition for an oriented near $\alpha$-valuation, one of the vertices $u,v$ is in $V_{large}$ and the other is in $V_{small}$. For each $k \in \{1,\ldots,n\}$, it is obtained as the label $p(v)-p(u)$ of precisely one directed edge $(u,v)$ of $G$.  Denote by $P$ the set of $k\in \{1,\ldots,n\}$ which arise as positive integer differences (i.e. from $(u,v)$ where $u \in V_{small}$ and $v \in V_{large}$), and by $N$ the set of $k \in \{1,\ldots,n\}$ which arise as negative integer differences (i.e. from $(u,v)$ where $u \in V_{large}$ and $v \in V_{small}$).  We work in the integers until the final step of the proof.

The first case (i.e. labels from $P$) is exactly the same as in the near $\alpha$ case. 
\begin{itemize}
\item We replace $v$ with a singleton set of vertices labelled by $V_v=\{lp(v)\}$.
\item We replace $u$ with the $l$-set of vertices labelled by $V_u = \bigcup_{i=0}^{l-1}\{lp(u)+i\}$.
\item The directed edge $(u,v)$ is replaced with the $l$ directed edges from $V_u$ to $V_v$; the corresponding differences for these edges are:
$lp(v)i-lp(u)-i =  l(p(v)-p(u))-i = lk-i$
where $0 \leq i \leq l-1$ and $k \in P$. 
\end{itemize}

Taking multiset unions of all of these differences:
\[
\bigcup_{\substack{(u,v) \in G \\ u \in V_{small}}} \left( \bigcup_{i=0}^{l-1} \{ (l(p(v)-p(u))-i) \} \right)
= \bigcup_{k\in P} \bigcup_{i=0}^{l-1} \{ lk-i \} = \bigcup_{k \in P}[l(k-1)+1,lk]. 
\]

Now consider the directed edges with labels corresponding to $N$. 
\begin{itemize}
\item We replace $u$ with a singleton set of vertices labelled by $V_u=\{lp(u)\}$.
\item We replace $v$ with the $l$-set of vertices labelled by $V_v = \bigcup_{i=0}^{l-1}\{lp(v)+i\}$.
\item  The directed edge $(u,v)$ is replaced with the $l$ directed edges from $V_u$ to $V_v$.  Since $p(v)-p(u)$ is a negative integer, the actual integer difference corresponding to $k \in N$ in the $\mathbb{Z}_{n+1}$ setting was in fact $-(n+1-k)$.  So the corresponding integer differences for these edges in the blow-up setting are:
\[lp(v)+i-lp(u)=l(p(v)-p(u))+i = l(-(n+1-k))+i=lk-(l-1)+i-(nl+1).\]
where $0 \leq i \leq l-1$ and $k \in N$.  
\end{itemize}
Taking multiset unions of all of these differences:
\[
\bigcup_{(u,v) \in G, v \in V_{small}} \left( \bigcup_{i=0}^{l-1} \{ (l(p(v)-p(u)-(n-1))+i) \} \right)
= \bigcup_{k\in N} \bigcup_{i=0}^{l-1} \{ lk-(l-1)+i-(nl+1) \}\]
\[= \bigcup_{k \in N}[l(k-1)+1-(nl+1),lk-(nl+1)] \equiv \bigcup_{k \in N}[l(k-1)+1,lk]\mod nl+1. \]

It is clear that all edges in the blow-up are formed by one of the two above cases. Taking the multiset union over $k \in P$ and $k \in N$ gives all $k \in \{1, \ldots ,n \}$, so (working now in $\mathbb{Z}_{nl+1}$):
\[\bigcup_{(x,y) \in G_{l,1}} \{p'(y)-p'(x)\} \equiv \bigcup_{k=1}^n[l(k-1)+1,lk] \mod nl+1 = \mathbb{Z}_{nl+1} \setminus \{0\}.\]

Now note that, since $p$ is an oriented near $\alpha$-valuation, for each $ y \in V_{large}$ and $x \in V_{small}$ we have $p(y) > p(x)$ (in the ordering $0<1< \cdots < n$). Hence in $G_{l,1}$ for any $s \in V_y$ and $t \in V_x$:
    \[p'(s) \geq p(y)l >   p(x)l = p'(t)\]
(in the ordering $0<1< \cdots< nl$).
Hence $p'$ is an oriented near $\alpha$-valuation.

\end{proof}

\begin{theorem}\label{lexiorblow}
If a digraph $G$ has an oriented near $\alpha$-valuation, then its lexicographic product with $K^c_l$ has an oriented near $\alpha$-valuation.
\end{theorem}
\begin{proof}
    This follows from Theorem \ref{orblowlarge}, Theorem \ref{orblowsmall} and Corollary \ref{lexiblow}.
\end{proof}

\begin{theorem}\label{thm:EDFfromorblowup}
Let $G$ be a digraph $G=(V,E)$ with an oriented near $\alpha$-valuation $p$.  Then there exists an $(|E|l^2 + 1, |V|, l, 1; G)$-EDF in $\mathbb{Z}_{|E|l^2+1}$.
\end{theorem}
\begin{proof}
    As $G$ has an oriented near $\alpha$-valuation, by Theorem \ref{lexiorblow}, so does $G \cdot K_{l}^C$; denote this oriented near $\alpha$-valuation by $p'$.  As this is an oriented near $\alpha$-valuation we have $\bigcup_{(x,y)\in \overrightarrow{E}(G \cdot K_{l}^C)} \{p'(y)-p'(x) \mod nl^2+1\} =\{1, \ldots nl^2\}$. Recall from the proof of Theorem \ref{orblowlarge} and Theorem \ref{orblowsmall} that each edge $(u,v) \in \overrightarrow{E}(G)$ corresponds to the $l^2$ edges from $V_u$ to $V_v$ in $G \cdot K_{l}^C$. Grouping the vertices of $G \cdot K_{l}^C$ according to these sets, and considering the corresponding edges between the sets in $G$ gives:
    \[\bigcup_{(u,v)\in \overrightarrow{E}(G)} \Delta(V_v,V_u ) = \bigcup_{(x,y)\in \overrightarrow{E}(G \cdot K_{l}^C)} \{p(y)-p(x)\}=\{1, \ldots nl^2\}=\mathbb{Z}_{nl^2+1}\backslash\{0\}.\]
\end{proof}

\begin{example}
We consider $H$ from Example \ref{oralp}, with $6$ vertices and $8$ edges. Using its oriented near $\alpha$-valuation, we can blow up the sets by a factor of $3$ and create a $(73, 6, 3, 1; H)$-EDF in $\mathbb{Z}_{73}$. The sets are (where for simplicity we use subscripts corresponding to the vertex label in each case): $V_0=\{0,3,6\}, V_{4} = \{36,39,42\}, V_{1} = \{9,12,15\},V_7 = \{63,62,61\}, V_8 = \{72,71,70\}$ and $V_5 = \{ 45,44,43\}$. Difference multisets are: $\Delta(V_5,V_4)=[1,9]$, $\Delta(V_1,V_8)=[10,18]$, $\Delta(V_7,V_4)=[19,27]$, $\Delta(V_8,V_4)=[28,36]$, $\Delta(V_5,V_0)=[37,45]$, $\Delta(V_7,V_1)=[46,54]$, $\Delta(V_7,V_0)=[55,63]$ and $\Delta(V_8,V_0)=[64,72]$.
\end{example}

\begin{remark}
For a graph $G$ with a near $\alpha$-valuation $a$, Theorem \ref{thm:nearalpha-blowup} yields a $G^a$-defined EDF with $\lambda=1$; taking the union of its directed edge-set $E$ and the edge-set obtained by reversing all edges in $E$, we obtain a $(n,m,l,2;G)$-EDF.  Similarly, for any digraph $H^*$ with an oriented near $\alpha$-valuation, denote by $H$ its underlying graph; by taking the edge-set of the $H^*$-defined EDF with $\lambda=1$ (obtained by Theorem \ref{thm:EDFfromorblowup}) and its reverse, we can obtain an $(n,m,l,2;H)-$EDF.
\end{remark}

\section{Digraph-defined EDFs with prescribed orientation}

So far in this paper, our priority has been to find \emph{some} orientation of our chosen underlying graph which yields a suitable digraph-defined EDF. We have allowed the orientation to be determined by the natural orientation of the vertex-labelling which we are using, or made adjustments to obtain the necessary differences.  However, in some situations a specific orientation is required, e.g. CEDFs are defined with a unidirectional (``clockwise") orientation around the cycle.  In \cite{PatStiCEDF}, combinations of techniques similar to those presented in the previous section are implicitly used to obtain $1$-CEDFs.  
We now show how we can use the results of the previous section to obtain a ``standard" orientation for some other natural classes of graphs, and hence corresponding digraph-defined EDFs.

We start by briefly considering unidirectional paths.  The vertex labelling for path $P_m$ (for even $m \geq 2$) given in Theorem \ref{pathalpha} gives rise to an oriented $\beta$-valuation with unidirectional orientation (stated in \cite{BloHsu}). The edge differences obtained using the natural orientation are $\{1,3,\ldots,m-1\}$ and $\{2,4,\ldots,m-2\}$; as the latter set is self-negative in $\mathbb{Z}_m$, we can apply Lemma \ref{flip} to reverse the orientation of all such edges to give the directed edge set for the unidirectional path $P_m$. We can use the blow-up process on this.

\begin{theorem}
Let $m \geq 2$ be even, and let the digraph $P_m^*$ be the unidirectional path with vertex-set $\{v_1,\ldots,v_m\}$ and directed edge-set $\{(v_1,v_2),\ldots,(v_{m-1},v_m)\}$.  Then there exists an $((m-1)l^2+1,m,l,1;P^*_m)$-EDF in $\mathbb{Z}_{(m-1)l^2+1}$, given by $\mathcal{A}=\{A_0,A_1,\ldots, A_{m-1}\}$
where
\begin{itemize}
\item $A_i = \{\frac{il^2}{2}+jl: 0 \leq j \leq l-1\}$ for $i$ even;
\item $A_i=\{(m-\frac{i+1}{2})l^2-j: 0 \leq j \leq l-1\}$ for $i$ odd.
\end{itemize}
\end{theorem}

\subsection{$2$-CEDFs}
A digraph-defined EDF where the digraph is a union of two or more same-size cycles, both oriented in a clockwise direction, is of particular interest since this corresponds to a $c$-CEDF with $c>1$, for which few constructions are currently known. In \cite{Wuetal}, a cyclotomic construction of $(q,m,2;1)$-$c$-CEDFs is given ($q$ a prime power), where the value of $c$ depends on certain conditions. In particular, we are aware of no direct explicit constructions for infinite families of $2$-CEDFs.  Since a $2$-CEDF with an odd number of sets is simply a $1$-CEDF, the case of interest is when the number of sets $m$ is even.  We will provide a construction which covers all possible 2-CEDF parameters $(n,m,l;1)$ when each of the two disjoint same-size cycles has an even number of sets (i.e. $4 \mid m$).

We first consider the case where each cycle-length is congruent to $0$ modulo $4$. We use a result of \cite{Abr} to establish a suitable oriented near $\alpha$-valuation.

\begin{example}\label{ex:2C4k}
Consider the disjoint union of two cycles, each of length $4k$: we denote this graph by $2C_{4k}$. We write the first cycle as $(v_0,v_1,\ldots,v_{4k-2},v_{4k-1})$ and the second as $(u_0,u_1,\ldots,u_{4k-2},u_{4k-1})$. An $\alpha$-valuation $a$ for this graph is given in \cite{Abr} for $k \geq 4$:
    \[a(v_i)=
    \begin{cases}
        8k-\frac{i}{2} & \text{if } i \text{ even, } 0 \leq i \leq 2k-{2}\\
        \frac{i-1}{2} & \text{if } i \text{ odd, } 1 \leq i \leq 2k-1\\
        5k+1 & \text{if } i = 2k\\
        8k-\frac{i}{2}+1 & \text{if } i \text{ even, } 2k+{2} \leq i \leq 4k-{2}\\
        \frac{i-1}{2} & \text{if } i \text{ odd, } 2k+1 \leq i \leq 4k-1\\
    \end{cases}\]
    \[a(u_i)=
    \begin{cases}
        6k+1-\frac{i}{2} & \text{if } i \text{ even, } 0 \leq i \leq 2k-2\\
        2k+1+\frac{i-1}{2} & \text{if } i \text{ odd, } 1 \leq i \leq 2k-1\\
        5k & \text{if } i = 2k\\
        6k-\frac{i}{2} & \text{if } i \text{ even, } 2k+2 \leq i \leq 4k-2\\
        2k+1+\frac{i-1}{2} & \text{if } i \text{ odd, } 2k+1 \leq i \leq 4k-1\\
    \end{cases}\]
    We will show that this can be used to obtain an oriented near $\alpha$-valuation for the oriented $2C_{4k}$ in which the orientation of each cycle is  clockwise, as in the CEDF definition. The bipartition property holds since this is an $\alpha$-valuation.

We first evaluate the edge-labels arising from the natural orientation induced by $a$:
We split this into cases:
\begin{itemize}
    \item [(i)] $\cup_{\substack{i \text{ even} \\ 2 \leq i \leq 2k-2}}\{a(v_i)-a(v_{i-1})\} = \cup_{\substack{i \text{ even} \\ 2 \leq i \leq 2k-2}}\{8k-\frac{i}{2}-\frac{(i-1)-1}{2}\}=\cup_{\substack{i \text{ even} \\ 2 \leq i \leq 2k-2}}\{8k-i+1\}=\{8k-1,8k-3,\ldots, 6k+5,6k+3\}$
    \item [(ii)]$\cup_{\substack{i \text{ even} \\ 0 \leq i \leq 2k-2}}\{a(v_i)-a(v_{i+1})\} = \cup_{\substack{i \text{ even} \\ 0 \leq i \leq 2k-2}}\{8k-\frac{i}{2}-\frac{(i+1)-1}{2}\}=\cup_{\substack{i \text{ even} \\ 0 \leq i \leq 2k-2}}\{8k-i\}=\{8k,8k-2,\ldots, 6k+4,6k+2\}$
    \item [(iii)]$a(v_{2k})-a(v_{2k-1})=5k+1-\frac{(2k-1)-1}{2}= 5k+1-k+1=\{4k+2\}$
    \item [(iv)]$a(v_{2k})-a(v_{2k+1})=5k+1-\frac{(2k+1)-1}{2}= 5k+1-k=\{4k+1\}$
    \item [(v)]$\cup_{\substack{i \text{ even} \\ 2k+2 \leq i \leq 4k-2}}\{a(v_i)-a(v_{i-1})\} = \cup_{\substack{i \text{ even} \\ 2k+2 \leq i \leq 4k-2}}\{8k-\frac{i}{2}+1-\frac{(i-1)-1}{2}\}=\cup_{\substack{i \text{ even} \\ 2k+2 \leq i \leq 4k-2}}\{8k-i+2\}=\{6k,6k-2,\ldots, 4k+6,4k+4\}$
    \item [(vi)]$\cup_{\substack{i \text{ even} \\ 2k+2 \leq i \leq 4k-2}}\{a(v_i)-a(v_{i+1})\} = \cup_{\substack{i \text{ even} \\ 2k+2 \leq i \leq 4k-2}}\{8k-\frac{i}{2}+1-\frac{(i+1)-1}{2}\}=\cup_{\substack{i \text{ even} \\ 2k+2 \leq i \leq 4k-2}}\{8k-i+1\}=\{6k-1,6k-3,\ldots, 4k+5,4k+3\}$
    \item [(vii)]$a(v_{0})-a(v_{4k-1})=8k-\frac{(4k-1)-1}{2}= \{6k+1\}$
\end{itemize}
Repeating the same process for $a(u_i)$'s we obtain the differences:
    \begin{itemize}
        \item [(viii)]$\{4k-1,4k-3,\ldots, 2k+5,2k+3\}$
        \item [(ix)]$\{4k,4k-2,\ldots,2k+4,2k+2\}$
        \item [(x)]$\{2k\}$
        \item [(xi)]$\{2k-1\}$
        \item [(xii)]$\{2k-2,2k-4,\ldots,4,2\}$
        \item [(xiii)] $\{2k-3,2k-5,\ldots,3,1\}$
        \item [(xiv)]$\{2k+1\}$
    \end{itemize}
We see that every non-zero element in $\{0,\ldots, 8k\}$ (which we will view as $\mathbb{Z}_{8k+1}$) appears precisely once as an edge-label.
We now consider orienting the edges of the two cycles in a clockwise direction. The edges currently not clockwise-oriented are precisely those with labels in cases (ii), (iv), (vi), (ix), (xi) and (xiii).

Observe that:
\begin{itemize}
\item $-\{8k,8k-2,\ldots, 6k+4,6k+2\} \equiv \{2k-3,2k-5,\ldots,3,1\} \cup \{2k-1\} \mod 8k+1$;
\item $-\{4k,4k-2,\ldots,2k+4,2k+2\} \equiv \{6k-1,6k-3,\ldots, 4k+5,4k+3\} \cup \{4k+1\} \mod 8k+1$.
\end{itemize}
Thus, when we reverse the orientation of the edges in the cases above, every edge is now clockwise-oriented and we still obtain each element in $\mathbb{Z}_{8k+1}$. By Lemma \ref{flip} $a$ is an oriented near $\alpha$-valuation for the clockwise orientation of $2C_{4k}$.
\end{example}

\begin{example}\label{ex:2C4ksmall}
In \cite{Abr}, $\alpha$-valuations are given for $2C_4,2C_8$ and $2C_{12}$ ($2C_{4k}$ with $1 \leq k \leq 3$).  The disjoint cycles are labelled as follows:
\begin{itemize}
\item[(i)] $2C_4$: $(0,8,1,6), (3,7,4,5)$;
\item[(ii)] $2C_8$: $(0,15,1,11,2,14,3,16),(5,12,6,10,7,9,8,13)$;
\item[(iii)] $2C_{12}$: $(0, 23, 1, 22, 2, 16, 3, 21, 4, 20, 5, 24), (7, 18, 8, 17,9, 15, 10, 14, 11, 13, 12, 19)$.
\end{itemize}
In Case (i), the differences under the natural orientation are $\{8,7,5,6\} \cup\{4,3,1,2\}$; reversing the orientation of every other (alternate) edge, the orientation is clockwise and the differences are $\{8,2,5,3\} \cup \{4,6,1,7\}=[1,8]$ modulo $9$.  Cases (ii) and (iii) are similar modulo $17$ and $25$ respectively.  Hence these are oriented near $\alpha$-valuations for digraphs $2C_4,2C_8$ and $2C_{12}$ with the clockwise orientation.
\end{example}

We now consider the case when the cycle lengths are congruent to $2$ modulo $4$.  Here we use an $\alpha$-valuation from \cite{Kot}.

\begin{example}\label{ex:2C4k+2}
Consider two disjoint cycles, each of length $4k+2$, we denote this $2C_{4k+2}$. We denote the first cycle by $(v_0,v_1,\ldots,v_{4k},v_{4k+1})$ and the second by $(u_0,u_1,\ldots,u_{4k},u_{4k+1})$. An $\alpha$-valuation $a$ for this graph is given in \cite{Kot} for $k \geq 2$:
    \[a(v_i)=
    \begin{cases}
        2k+2+\frac{i}{2} & \text{if } i \text{ even } 0 \leq i \leq 2k-2\\
        6k+3-\frac{i-1}{2} & \text{if } i \text{ odd } 1 \leq i \leq 2k-3\\
        6k+3+\frac{i+1}{2} & \text{if } i \text{ odd } 2k-1 \leq i \leq 4k+1\\
        2k-\frac{i}{2} & \text{if } i \text{ even } 2k \leq i \leq 4k\\
    \end{cases}\]
    \[a(u_i)=
    \begin{cases}
        5k+4-\frac{i}{2} & \text{if } i \text{ even } 0 \leq i \leq 2k+2\\
        3k+2+\frac{i-1}{2} & \text{if } i \text{ odd } 1 \leq i \leq 2k+1\\
        3k+1-\frac{i-1}{2} & \text{if } i \text{ odd } 2k+3 \leq i \leq 4k+1\\
        5k+2+\frac{i}{2} & \text{if } i \text{ even } 2k+4 \leq i \leq 4k\\
    \end{cases}\]
    \begin{itemize}
    \item [(i)] $\cup_{\substack{i \text{ odd} \\ 1 \leq i \leq 2k-3}}\{a(v_i)-a(v_{i-1})\} =\{4k+1,4k-1,4k-3, \ldots 2k+9,2k+7,2k+5\}$
    \item [(ii)]$\cup_{\substack{i \text{ odd } \\ 1 \leq i \leq 2k-3}}\{a(v_i)-a(v_{i+1})\} =\{4k,4k-2,4k-4,\ldots, 2k+6,2k+4\}$
    \item [(iii)]$a(v_{2k-1})-a(v_{2k-2})=6k+3+\frac{(2k-1)+1}{2}-(2k+2+\frac{2k-2}{2})=\{4k+2\}$
    \item [(iv)]$\cup_{\substack{i \text{ odd} \\ 2k+1 \leq i \leq 4k+1}}\{a(v_i)-a(v_{i-1})\} =\{6k+4,6k+6,\ldots, 8k+2,8k+4\}$
    \item [(v)]$\cup_{\substack{i \text{ odd} \\ 2k-1 \leq i \leq 4k-1}}\{a(v_i)-a(v_{i+1})\} = \{6k+3,6k+5,\ldots, 8k+1,8k+3\}$
    \item [(vi)]$a(v_{4k+1})-a(v_{0})=6k+3+\frac{(4k+1)+1}{2} -(2k+2)= \{6k+2\}$
\end{itemize}
Repeating this process for $a(u_i)$'s, we obtain the differences:
    \begin{itemize}
        \item [(vii)]$\{1,3,\ldots, 2k-1,2k+1\}$
        \item [(viii)]$\{2,4,\ldots, 2k,2k+2\}$
        \item [(ix)]$\{2k+3\}$
        \item [(x)]$\{6k,6k-2,\ldots,4k+6,4k+4\}$
        \item [(xi)]$\{6k+1,\ldots 4k+7,4k+5\}$
        \item [(xii)]$\{4k+3\}$
    \end{itemize}
    We see that every non-zero element in $\{0,\ldots, 8k+4\}$ (which we will view as $\mathbb{Z}_{8k+5}$) appears precisely once as an edge-label.
    We now consider orienting the edges of the two cycles in a clockwise direction. The edges currently not clockwise-oriented are precisely those in cases (ii), (v), (vi), (viii), (ix) and (xi).

Observe that:
\begin{itemize}
\item $-\{4k,4k-2, \ldots, 2k+4\} \equiv \{6k+1,\ldots,4k+7,4k+5\} \mod 8k+5$
\item $-\{6k+3,6k+5,\ldots, 8k+1,8k+3\} \equiv \{2,4,\ldots, 2k,2k+2\} \mod 8k+5$
\item $-(6k+2) \equiv 2k+3 \mod 8k+5$
\end{itemize}
Thus, when we reverse the orientation of the edges above, every edge is now clockwise-oriented and we still obtain each element in $\mathbb{Z}_{8k+5}$. By Lemma \ref{flip} $a$ is an oriented near $\alpha$-valuation for the clockwise orientation of $2C_{4k+2}$.
\end{example}

\begin{example}\label{ex:2C6}
In \cite{Kot}, an $\alpha$-valuation is given for $2C_6$.  This leads to the following oriented near $\alpha$-valuation for the clockwise orientation of $2C_6$: $(0,11,1,10,4,12), (5,9,2,7,6,8)$.
\end{example}

This leads to the following theorem on $2$-CEDFs:
\begin{theorem}\label{thm:2CEDF}
Let $k,l \geq 1$.  Then there exists an $(4kl^2+1,4k,l,1)$-$2$-CEDF in $\mathbb{Z}_{4kl^2+1}$.
\end{theorem}
\begin{proof}
Let $2C_{2k}$ be the disjoint union of two cycles of length $2k$, and let $2C_{2k}^*$ be the oriented $2C_{2k}$ with both oriented cycles $C$ and $C'$  having clockwise orientation. Then by Example \ref{ex:2C4k}, Example \ref{ex:2C4ksmall}, Example \ref{ex:2C4k+2} and Example \ref{ex:2C6}, there exists an oriented near $\alpha$-valuation for $2C_{2k}^*$.  Applying the blow-up construction of Theorem \ref{thm:EDFfromorblowup} yields an $(4kl^2 +1,4k,l,1;2C_{2k}^*)$-EDF in $\mathbb{Z}_{4kl^2+1}$ with ordered sets $(A_1,\ldots,A_k)$ corresponding to $C$, and $(B_1,B_2,\ldots,B_k)$ corresponding to $C'$.  Interleaving these as $(A_1,B_1,A_2,B_2,\ldots,A_k,B_k)$ gives the required $(4kl^2+1,4k,l,1)$-$2$-CEDF in $\mathbb{Z}_{4kl^2+1}$.
\end{proof}

\begin{example}
We demonstrate how the oriented near $\alpha$-valuation for the clockwise-oriented $2C_4$ (which we denote $2C_4^*)$ given in Example \ref{ex:2C4ksmall} yields a $(73,8,3,1;2C_4^*)$-EDF, i.e. a $(73,8,3,1)$-$2$-CEDF, in $\mathbb{Z}_{73}$.  We apply Theorem \ref{thm:2CEDF}.  The cycle $(0,8,1,6)$ gives ordered sets $(A_1,A_2,A_3,A_4)$ where
$A_1=\{0,3,6\}$, $A_2=\{70,71,72\}$, $A_3=\{9,12,15\}$ and $A_4=\{52,53,54\}$. Clockwise pairwise differences are $[64,72], [10,18], [37,45]$ and $[19,27]$. The cycle $(3,7,4,5)$ gives ordered sets $(B_1,B_2,B_3,B_4)$ where
$B_1=\{27,30,33\}$, $B_2=\{61,62,63\}$, $B_3= \{36,39,42\}$ and $B_4=\{43,44,45\}$. Clockwise pairwise differences are $[28,36],[46,54],[1,9]$ and $[55,63]$.  Interleaving these sets gives the $(73,8,3,1)$-$2$-CEDF $(A_1,B_1,A_2,B_2,A_3,B_3,A_4,B_4)$.
\end{example}

Another family of graphs related to cycles is the family of sun graphs.  A sun graph $S_m$ ($m$ even) is a cycle of length $\frac{m}{2}$ with a pendant edge connected to each vertex of the cycle.  In the Appendix, we give a new $\alpha$-valuation for sun graphs, which is also an oriented $\alpha$-valuation for what we will call a semi-directed sun digraph - a unidirectional oriented cycle with pendant edges oriented alternately inward and outward. We can use this to create digraph defined EDFs as follows (proof omitted).

\begin{theorem}
    Let $k,l \geq 1$. Let $S_{8k}^*$ be the semi-directed sun digraph with $8k$ vertices. Then there exists a $(8kl^2+1,8k,l,1;S_{8k}^*)-EDF$ in $\mathbb{Z}_{8kl^2+1}$.
\end{theorem}

\subsection{Ladder graphs}
Finally, we present EDFs from ladder graphs with the ``natural" orientation.

\begin{example}
Denote by $L_{2k+1}$ the ladder graph with $2k+1$ rungs, vertices $\{u_i,v_i: 0 \leq i \leq 2k\}$ and edges $\{\{u_i,v_i\}: 0 \leq i \leq 2k\}\cup \{\{u_i,u_{i+1}\}: 0 \leq i \leq 2k-1\}\cup \{\{v_i,v_{i+1}\}: 0 \leq i \leq 2k-1\}$. 
The following $\alpha$-valuation is given in \cite{Maheo}:
\begin{itemize}
\item $a(u_{2i})=2k+i$, $a(v_{2i})=6k+1-i$  ($0 \leq i \leq k$);
\item $a(u_{2i+1})=4k-i$, $a(v_{2i+1})=i$ ($0 \leq i \leq k-1$).
\end{itemize}
We will show that this can be adapted to obtain an oriented near $\alpha$-valuation with the orientation of the ladder ``standardised" to be left-to-right and bottom-to-top (working in $\mathbb{Z}_{6k+2}$). The bipartition property holds since this is an $\alpha$-valuation.

We first evaluate the edge-labels arising from the natural orientation induced by $a$:
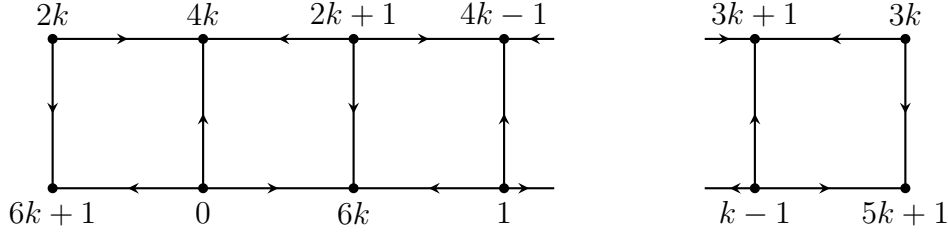
    \begin{center}
    \begin{tikzpicture}[baseline=4mm,scale=0.66]
    \begin{scope}[radius=1mm]
    \fill(0,2) circle;
    \fill(0,-1) circle; 
    \fill(6,-1) circle;
    \fill(6,2) circle; 
    \fill(3,-1) circle;
    \fill(3,2) circle;
    \fill(9,-1) circle;
    \fill(9,2) circle;
    \fill(14,-1) circle;
    \fill(14,2) circle;
    \fill(17,-1) circle;
    \fill(17,2) circle;
    
    \node at (17,2.5) {$3k$};
    \node at (17,-1.5) {$5k+1$};
    \node at (14,2.5) {$3k+1$};
    \node at (14,-1.5) {$k-1$};
    \node at (9,2.5) {$4k-1$};
    \node at (9,-1.5) {$1$};
    \node at (6,2.5) {$2k+1$};
    \node at (6,-1.5) {$6k$};
    \node at (3,2.5) {$4k$};
    \node at (3,-1.5) {$0$};
    \node at (0,2.5) {$2k$};
    \node at (0,-1.5) {$6k+1$};
    \end{scope}
    \begin{scope}[thick]
    \path [draw=black,postaction={on each segment={mid arrow=black}}](0,2)--(3,2);
    \path [draw=black,postaction={on each segment={mid arrow=black}}](3,-1)--(0,-1);
    \path [draw=black,postaction={on each segment={mid arrow=black}}](6,2)--(3,2);
    \path [draw=black,postaction={on each segment={mid arrow=black}}](3,-1)--(6,-1);
    \path [draw=black,postaction={on each segment={mid arrow=black}}](6,2)--(9,2);
    \path [draw=black,postaction={on each segment={mid arrow=black}}](9,-1)--(6,-1);
    \path [draw=black,postaction={on each segment={mid arrow=black}}](17,2)--(14,2);
    \path [draw=black,postaction={on each segment={mid arrow=black}}](14,-1)--(17,-1);
    
    \path [draw=black,postaction={on each segment={mid arrow=black}}](10,2)--(9,2);
    \path [draw=black,postaction={on each segment={mid arrow=black}}](9,-1)--(10,-1);
    
    \path [draw=black,postaction={on each segment={mid arrow=black}}](13,2)--(14,2);
    \path [draw=black,postaction={on each segment={mid arrow=black}}](14,-1)--(13,-1);
    
    \path [draw=black,postaction={on each segment={mid arrow=black}}](0,2)--(0,-1);
    \path [draw=black,postaction={on each segment={mid arrow=black}}](3,-1)--(3,2);
    \path [draw=black,postaction={on each segment={mid arrow=black}}](6,2)--(6,-1);
    \path [draw=black,postaction={on each segment={mid arrow=black}}](9,-1)--(9,2);
    \path [draw=black,postaction={on each segment={mid arrow=black}}](14,-1)--(14,2);
    \path [draw=black,postaction={on each segment={mid arrow=black}}](17,2)--(17,-1);

    \end{scope}
    \end{tikzpicture}
    \end{center}

We split this into cases:
\begin{itemize}
\item[(i)] $\bigcup_{0 \leq i \leq k}\{v_{2i}-u_{2i}\}=\bigcup_{0 \leq i \leq k}\{6k+1-i-(2k+i)\}=\bigcup_{0 \leq i \leq k}\{4k-2i+1\}\\
=\{4k+1,4k-1,4k-3,\ldots, 2k+5,2k+3,2k+1\}$.
\item[(ii)] $\bigcup_{0 \leq i \leq k-1}\{u_{2i+1}-v_{2i+1}\}=\bigcup_{0 \leq i \leq k-1}\{4k-i-i\}=\bigcup_{0 \leq i \leq k-1}\{4k-2i\}\\=\{4k,4k-2,4k-4,\ldots, 2k+6,2k+4,2k+2\}$.
\item[(iii)] $\bigcup_{0 \leq i \leq k-1}\{v_{2i}-v_{2i+1}\}=\bigcup_{0 \leq i \leq k-1}\{6k+1-i-i\}=\bigcup_{0 \leq i \leq k-1}\{6k-2i+1\}\\
=\{6k+1,6k-1,6k-3,\ldots, 4k+7,4k+5,4k+3\}$.
\item[(iv)] $\bigcup_{0 \leq i \leq k-1}\{u_{2i+1}-u_{2i}\}=\bigcup_{0 \leq i \leq k-1}\{4k-i-(2k+i)\}=\bigcup_{0 \leq i \leq k-1}\{2k-2i\}\\
=\{2k,2k-2,2k-4,\ldots, 6,4,2\}$.
\item[(v)]$\bigcup_{1 \leq i \leq k}\{v_{2i}-v_{2(i-1)+1}\}=\bigcup_{1 \leq i \leq k}\{6k+1-i-(i-1)\}=\bigcup_{1 \leq i \leq k}\{6k-2i+2\}\\
=\{6k,6k-2,6k-4,\ldots, 4k+6,4k+4,4k+2\}$.
\item[(vi)] $\bigcup_{1 \leq i \leq k}\{u_{2(i-1)+1}-u_{2i}\}=\bigcup_{1 \leq i \leq k}\{4k-(i-1)-(2k+i)\}=\bigcup_{1 \leq i \leq k}\{2k-2i+1\}\\
=\{2k-1,2k-3,2k-5,\ldots, 5,3,1\}$.
\end{itemize} 
We see that every non-zero element in $\mathbb{Z}_{6k+2}$ appears precisely once as an edge-label.
    
We now consider orienting the edges of the graph from left-to-right and bottom-to-top:
    
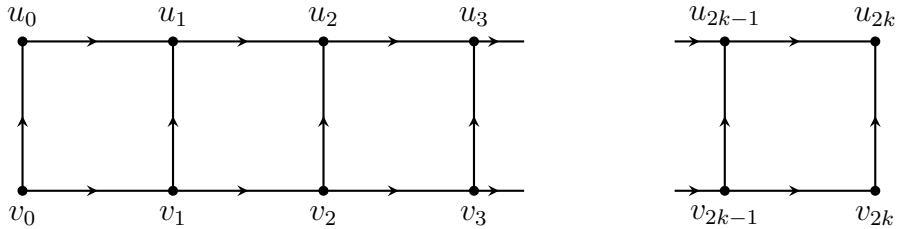
\begin{center}
    \begin{tikzpicture}[baseline=4mm,scale=0.66]
    \begin{scope}[radius=1mm]
    \fill(0,2) circle;
    \fill(0,-1) circle; 
    \fill(6,-1) circle;
    \fill(6,2) circle; 
    \fill(3,-1) circle;
    \fill(3,2) circle;
    \fill(9,-1) circle;
    \fill(9,2) circle;
    \fill(14,-1) circle;
    \fill(14,2) circle;
    \fill(17,-1) circle;
    \fill(17,2) circle;
    
    \node at (17,2.5) {$u_{2k}$};
    \node at (17,-1.5) {$v_{2k}$};
    \node at (14,2.5) {$u_{2k-1}$};
    \node at (14,-1.5) {$v_{2k-1}$};
    \node at (9,2.5) {$u_{3}$};
    \node at (9,-1.5) {$v_{3}$};
    \node at (6,2.5) {$u_{2}$};
    \node at (6,-1.5) {$v_{2}$};
    \node at (3,2.5) {$u_{1}$};
    \node at (3,-1.5) {$v_{1}$};
    \node at (0,2.5) {$u_{0}$};
    \node at (0,-1.5) {$v_{0}$};
    \end{scope}
    \begin{scope}[thick]
    \path [draw=black,postaction={on each segment={mid arrow=black}}](0,2)--(3,2);
    \path [draw=black,postaction={on each segment={mid arrow=black}}](0,-1)--(3,-1);
    \path [draw=black,postaction={on each segment={mid arrow=black}}](3,2)--(6,2);
    \path [draw=black,postaction={on each segment={mid arrow=black}}](3,-1)--(6,-1);
    \path [draw=black,postaction={on each segment={mid arrow=black}}](6,2)--(9,2);
    \path [draw=black,postaction={on each segment={mid arrow=black}}](6,-1)--(9,-1);
    \path [draw=black,postaction={on each segment={mid arrow=black}}](14,2)--(17,2);
    \path [draw=black,postaction={on each segment={mid arrow=black}}](14,-1)--(17,-1);
    
    \path [draw=black,postaction={on each segment={mid arrow=black}}](9,2)--(10,2);
    \path [draw=black,postaction={on each segment={mid arrow=black}}](9,-1)--(10,-1);
    
    \path [draw=black,postaction={on each segment={mid arrow=black}}](13,2)--(14,2);
    \path [draw=black,postaction={on each segment={mid arrow=black}}](13,-1)--(14,-1);
    
    \path [draw=black,postaction={on each segment={mid arrow=black}}](0,-1)--(0,2);
    \path [draw=black,postaction={on each segment={mid arrow=black}}](3,-1)--(3,2);
    \path [draw=black,postaction={on each segment={mid arrow=black}}](6,-1)--(6,2);
    \path [draw=black,postaction={on each segment={mid arrow=black}}](9,-1)--(9,2);
    \path [draw=black,postaction={on each segment={mid arrow=black}}](14,-1)--(14,2);
    \path [draw=black,postaction={on each segment={mid arrow=black}}](17,-1)--(17,2);

    \end{scope}
    \end{tikzpicture}
    \end{center}

Note that, 
\begin{itemize}
\item for the alternate vertical edges from case (i):
$\{4k+1,4k-1,\ldots,2k+3,2k+1\} \equiv -\{4k+1,4k-1,\ldots,2k+3,2k+1\} \mod 6k+2$;
\item for alternate horizontal edges from cases (iii) and (vi): $\{6k+1,6k-1,6k-3,\ldots, 4k+7,4k+5,4k+3\} \equiv -\{2k-1,2k-3,\ldots,3,1\} \mod 6k+2$.
\end{itemize}
Hence, applying Lemma \ref{flip}, $a$ is also an oriented $\alpha$-valuation for the desired orientation.
\end{example}

\begin{theorem}
Let $k,l \geq 1$. Let $L_{2k+1}$ be the ladder graph with $2k+1$ rungs, and let $L_{2k+1}^*$ be the directed $L_{2k+1}$ with left-to-right and bottom-to-top orientation.  Then there exists a $((6k+1)l^2 +1,4k+2,l,1;L_{2k+1}^*)$-EDF in $\mathbb{Z}_{(6k+1)l^2 +1}$.
\end{theorem}

\begin{remark}
It is not always possible to apply this orientation-standardisation approach starting from the natural orientation of a given near $\alpha$-valuation.  An $\alpha$-valuation for the square lattice graphs is given in \cite{AchGil}, but our techniques cannot be applied to obtain an oriented near $\alpha$-valuation with a left-to-right bottom-to-top orientation; similarly, the $\alpha$-valuation given for symmetric trees in \cite{Rob} cannot be adjusted to obtain a root-down orientation. This is due to the presence of pairs of edges labelled $\{e,-e\}$, positioned in the graph in such a way that the orientations of both cannot be suitably reversed.
\end{remark}

\section{Further Work}
In this paper, we have explored how suitable vertex-labellings of graphs and digraphs can lead directly to constructions for digraph-directed EDFs in cyclic groups. We present some open questions, both about the labellings themselves and the digraph-defined EDFs.

\begin{itemize}
    \item Any graph which allows a near-$\alpha$ valuation must be bipartite.  Does every bipartite graph that possesses a $\beta$-valuation have a near $\alpha$-valuation? We have examples of $\beta$-valuations on bipartite graphs which are not near $\alpha$-valuations, but for each there exists a different vertex labelling which is a near $\alpha$-valuation.
    \item From Theorem \ref{nearalphaprod}, the weak tensor product of any two graphs with near-$\alpha$ valuations is a graph with a near $\alpha$-valuation. For two graphs which each possess a near $\alpha$ valuation but no $\alpha$-valuation, will their weak tensor product have no $\alpha$-valuation?
    \item In this paper, we have created digraph-defined EDFs with $\lambda = 1$ and - by taking union of the set of oriented edges and its reverse - in some cases $\lambda=2$.  Can the graph valuation/blow-up method be used to obtain such EDFs with $\lambda \neq 1,2$?
    \item Our method for obtaining digraph-defined EDFs requires all underlying graphs to be bipartite.  However, $\beta$-valuations can exist for non-bipartite graphs.  Can a similar ``valuation then blow-up" method be developed for non-bipartite graphs? 
\end{itemize}

\subsection*{Acknowledgements}
The work of the first author was supported by the London Mathematical Society grant URB-2025-54.

\section{Appendix: sun graphs}

In this Appendix, we present an $\alpha$-valuation for a class of graphs closely related to cycles, namely sun graphs $S_m$.  We believe this is a new contribution to the $\alpha$-valuation literature. 

\begin{definition}
    A \emph{sun graph} $S_m$ on $m$ vertices ($m$ even) is a cycle of length $\frac{m}{2}$ with a degree one vertex connected to each vertex of the cycle. 
\end{definition}

\begin{example}
    We give an example of $S_{24}$ with a near $\alpha$-valuation in Figure \ref{sg}. 
    
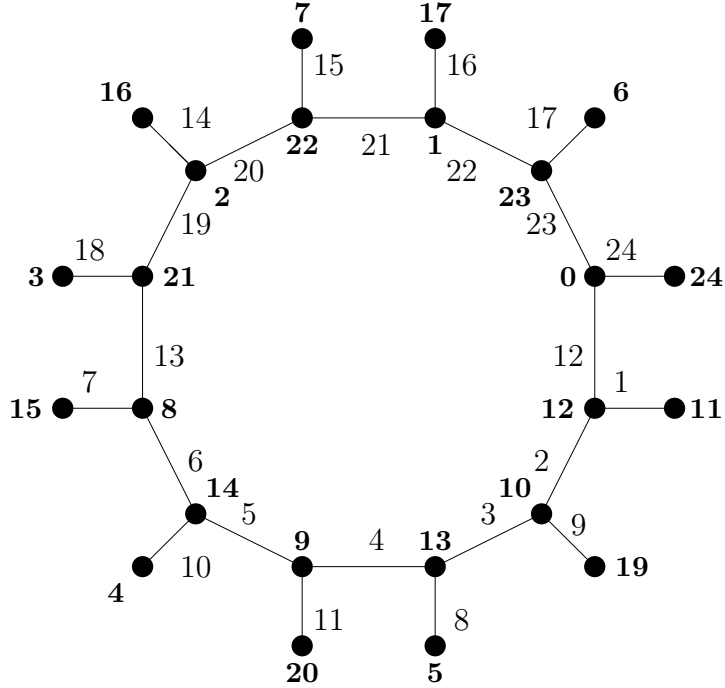
\begin{figure}\label{sg}
    \begin{center}
        \begin{tikzpicture}[baseline=4mm,scale=1.4]
            \begin{scope}[radius=1mm]
                \fill(7.5,5.75) circle;
                \fill(8,5.25) circle;
                \fill(9,5.75) circle;
                \fill(9,6.5) circle;
                \fill(10.25,6.5) circle;
                \fill(10.25,5.75) circle;
                \fill(11.25,5.25) circle;
                \fill(11.75,5.75) circle;
                \fill(12.5,4.25) circle;
                \fill(11.75,4.25) circle;
                \fill(11.75,3) circle;
                \fill(12.5,3) circle;
                \fill(11.25,2) circle;
                \fill(11.75,1.5) circle;
                \fill(10.25,1.5) circle;
                \fill(10.25,0.75) circle;
                \fill(9,0.75) circle;
                \fill(9,1.5) circle;
                \fill(7.5,1.5) circle;
                \fill(8,2) circle;
                \fill(6.75,3) circle ;
                \fill(6.75,4.25) circle;
                \fill(7.5,4.25) circle;
                \fill(7.5,3) circle;
            \end{scope}
            \begin{scope}[thick]
                \draw [line width=0.2pt] (7.5,4.25) -- (7.5,3);
                \draw [line width=0.2pt] (7.5,4.25) -- (8,5.25);
                \draw [line width=0.2pt] (8,5.25) -- (9,5.75);
                \draw [line width=0.2pt] (9,5.75) -- (10.25,5.75);
                \draw [line width=0.2pt] (10.25,5.75) -- (11.25,5.25);
                \draw [line width=0.2pt] (11.25,5.25) -- (11.75,4.25);
                \draw [line width=0.2pt] (11.75,4.25) -- (11.75,3);
                \draw [line width=0.2pt] (7.5,3) -- (8,2);
                \draw [line width=0.2pt] (8,2) -- (9,1.5);
                \draw [line width=0.2pt] (9,1.5) -- (10.25,1.5);
                \draw [line width=0.2pt] (10.25,1.5) -- (11.25,2);
                \draw [line width=0.2pt] (11.25,2) -- (11.75,3);
                \draw [ line width=0.2pt](9,5.75) to (9,6.5);
                \draw [ line width=0.2pt](10.25,5.75) to (10.25,6.5);
                \draw [ line width=0.2pt](11.25,5.25) to (11.75,5.75);
                \draw [ line width=0.2pt](11.75,4.25) to (12.5,4.25);
                \draw [ line width=0.2pt](11.75,3) to (12.5,3);
                \draw [ line width=0.2pt](11.25,2) to (11.75,1.5);
                \draw [ line width=0.2pt](10.25,1.5) to (10.25,0.75);
                \draw [ line width=0.2pt](9,1.5) to (9,0.75);
                \draw [ line width=0.2pt](8,2) to (7.5,1.5);
                \draw [ line width=0.2pt](7.5,3) to (6.75,3);
                \draw [ line width=0.2pt](7.5,4.25) to (6.75,4.25);
                \draw [ line width=0.2pt](8,5.25) to (7.75,5.5);
                \draw [ line width=0.2pt](8,5.25) to (7.5,5.75);
                \node [font=\small] at (9,6.75) {\textbf{7}};
                \node [font=\small] at (10.25,6.75) {\textbf{17}};
                \node [font=\small] at (12,6) {\textbf{6}};
                \node [font=\small] at (12.8,4.25) {\textbf{24}};
                \node [font=\small] at (12.8,3) {\textbf{11}};
                \node [font=\small] at (12.1,1.5) {\textbf{19}};
                \node [font=\small] at (10.25,0.5) {\textbf{5}};
                \node [font=\small] at (9,0.5) {\textbf{20}};
                \node [font=\small] at (7.25,1.25) {\textbf{4}};
                \node [font=\small] at (6.4,3) {\textbf{15}};
                \node [font=\small] at (6.5,4.25) {\textbf{3}};
                \node [font=\small] at (7.75,3) {\textbf{8}};
                \node [font=\small] at (7.85,4.25) {\textbf{21}};
                \node [font=\small] at (8.25,5) {\textbf{2}};
                \node [font=\small] at (9,5.5) {\textbf{22}};
                \node [font=\small] at (10.25,5.5) {\textbf{1}};
                \node [font=\small] at (11,5) {\textbf{23}};
                \node [font=\small] at (11.5,4.25) {\textbf{0}};
                \node [font=\small] at (11.4,3) {\textbf{12}};
                \node [font=\small] at (8.25,2.25) {\textbf{14}};
                \node [font=\small] at (9,1.75) {\textbf{9}};
                \node [font=\small] at (10.25,1.75) {\textbf{13}};
                \node [font=\small] at (11,2.25) {\textbf{10}};
                \node [font=\small] at (7.25,6) {\textbf{16}};
                \node  at (8,1.5) {10};
                \node  at (7,3.25) {7};
                \node  at (7,4.5) {18};
                \node  at (8,5.75) {14};
                \node  at (9.25,6.25) {15};
                \node  at (10.5,6.25) {16};
                \node  at (11.25,5.75) {17};
                \node  at (12,4.5) {24};
                \node  at (12,3.25) {1};
                \node  at (11.6,1.9) {9};
                \node  at (10.5,1) {8};
                \node  at (9.25,1) {11};
                \node  at (9.7,5.5) {21};
                \node  at (10.5,5.25) {22};
                \node  at (11.25,4.75) {23};
                \node  at (11.5,3.5) {12};
                \node  at (11.25,2.5) {2};
                \node  at (10.75,2) {3};
                \node  at (9.7,1.75) {4};
                \node  at (8.5,2) {5};
                \node  at (8,2.5) {6};
                \node  at (7.75,3.5) {13};
                \node  at (8,4.75) {19};
                \node  at (8.5,5.25) {20};
            \end{scope}
        \end{tikzpicture}
    \end{center}
    \caption{$S_{24}$, with a near $\alpha$-valuation.}
\end{figure}
\end{example}

\begin{theorem}
    Let $n=8k$ for $k\geq 1$, and let $G$ be the sun graph on $m$ vertices. We label the vertices in the cycle $v_i$ and label the outer vertices $u_i$ for $0 \leq i \leq 4k-1$. Define:
    \[
    a(v_i)=
    \begin{cases}
        i/2 & \text{if } i \text{ even } 0 \leq i \leq 2k-2\\
        8k-(i+1)/2 & \text{if } i \text{ odd } 1 \leq i \leq 2k-1\\
        i/2 +2k-1 & \text{if } i \text{ even } 2k \leq i \leq 4k-2\\
        6k - (i+1)/2 & \text{if } i \text{ odd } 2k+1 \leq i \leq 4k-1\\
    \end{cases}
    \]
    \[
    a(u_i)=
    \begin{cases}
        8k & \text{if } i =0\\
        6k-i/2 & \text{if } i \text{ even } 2 \leq i \leq 2k-2\\
        2k+(i-1)/2 & \text{if } i \text{ odd } 1 \leq i \leq 2k-3\\
        k & \text{if } i =2k-1\\
        5k & \text{if } i =2k\\
        8k-i/2 & \text{if } i \text{ even } 2k+2 \leq i \leq 4k-2\\
        (i+1)/2 & \text{if } i \text{ odd } 2k+1 \leq i \leq 4k-3\\
        4k-1 & \text{if } i =4k-1\\
    \end{cases}
    \]
where, in the case-splits, only intervals $a \leq i \leq b$ such that $a < b$ for the given value of $k$ are considered.  Then $a$ is an $\alpha$-valuation of $G$.
\end{theorem}
\begin{proof}
To check the first condition, we consider the absolute values of the differences of the vertex labels. We consider the cycle and the other edges separately.  First we consider the cycle; we split this into the odd/even vertices smaller than $2k$, and those larger or equal to $2k$. Note that $a(v_{2k-1})-a(v_{2k})$ falls between these, so is considered separately.

\begin{align*}
&  \left( \bigcup^{4k-2}_{i=0}\{|a(v_i)-a(v_{i+1})|\} \right) \cup \{|a(v_{4k-1})-a(v_{0})|\}\\
&= \bigcup_{\substack{i \text{ even } \\ 0 \leq i \leq 2k-2}} \left\{\left|\frac{i}{2}-8k+\frac{i+2}{2}\right|\right\} \cup  \bigcup_{\substack{i \text{ odd } \\ 1 \leq i \leq 2k-3}}\left\{\left|8k-\frac{i+1}{2}-\frac{i+1}{2}\right|\right\} \\
&\cup \left\{\left|8k-\frac{2k-1+1}{2}-\frac{2k}{2}-2k+1\right|\right\} \cup \bigcup_{\substack{i \text{ even } \\ 2k \leq i \leq 4k-2}}\left\{\left|\frac{i}{2}+2k-1-6k+\frac{i+2}{2}\right|\right\}\\
&\cup \bigcup_{\substack{i \text{ odd } \\ 2k+1 \leq i \leq 4k-3}}\left\{\left|6k-\frac{i+1}{2}-\frac{i+1}{2}-2k+1\right|\right\} \cup \left\{\left|6k-\frac{(4k-1)+1}{2}-0\right|\right\}\\
&= \bigcup_{\substack{i \text{ even } \\ 0 \leq i \leq 2k-2}}\{8k-i-1\}  \cup \bigcup_{\substack{i \text{ odd } \\ 1 \leq i \leq 2k-3}}\{8k-i-1\} \cup\{4k+1\}\cup \bigcup_{\substack{i \text{ even } \\ 2k \leq i \leq 4k-2}}\{4k-i\}\\
&\cup \bigcup_{\substack{i \text{ odd } \\ 2k+1 \leq i \leq 4k-3}}\{4k-i\} \cup\{4k\}\\
&=[6k+1,8k-1]\cup\{4k+1\}\cup[2,2k]\cup\{4k\}
\end{align*}

Now we consider the other edges; again we split this into several ranges.
\begin{align*}
& \bigcup^{4k-1}_{i=0}\{|a(v_i)-a(u_i)|\}= \{|0-8k|\} \cup \bigcup_{\substack{i \text{ even } \\ 0 \leq i \leq 2k-2}}\left\{\left|\frac{i}{2}-6k+\frac{i}{2}\right|\right\} \cup \bigcup_{\substack{i \text{ odd } \\ 1 \leq i \leq 2k-3}}\left\{\left|8k-\frac{i+1}{2}-2k-\frac{i-1}{2}\right|\right\}\\
& \cup \left\{\left|8k-\frac{(2k-1)+1}{2}-k\right|\right\} \cup \left\{\left|\frac{2k}{2}+2k-1-5k\right|\right\} \cup \bigcup_{\substack{i \text{ even } \\ 2k+2 \leq i \leq 4k-2}}\left\{\left|\frac{i}{2}+2k-1-8k+\frac{i}{2}\right|\right\}\\
&\cup \bigcup_{\substack{i \text{ odd } \\ 2k+1 \leq i \leq 4k-3}}\left\{\left|6k-\frac{i+1}{2}-\frac{i+1}{2}\right|\right\} \cup \left\{\left|6k-\frac{(4k-1)+1}{2}-4k-1\right|\right\}\\
&=\{8k\}\cup \bigcup_{\substack{i \text{ even } \\ 0 \leq i \leq 2k-2}}\{6k-i\} \cup \bigcup_{\substack{i \text{ odd } \\ 1 \leq i \leq 2k-3}}\{6k-i\} \cup\{6k\}\cup\{2k+1\}\cup \bigcup_{\substack{i \text{ even } \\ 2k+2 \leq i \leq 4k-2}}\{6k-i+1\} \\& \cup \bigcup_{\substack{i \text{ odd } \\ 2k+1 \leq i \leq 4k-3}}\{6k-i-1\}\cup\{1\}\\
& =\{8k\}\cup[4k+2,6k-1]\cup\{6k,2k+1\}\cup[2k+2,4k-1]\cup\{1\}
\end{align*}

Combining both of these gives $[1,8k]=\mathbb{Z}_{8k+1}\backslash\{0\}$.  

For the second condition: in the cycle, note that the labels of $v_i$ ($i$ even) are at most $4k-2$, while those of $v_i$ ($i$ odd) are at least $4k$. For the pendant vertices, the labels of the $u_i$ ($i$ even) are at least $5k$ while the labels of the $u_i$ ($i$ odd) are at most $4k-1$. So this is an $\alpha$-valuation with $x=4k-1$.
\end{proof}

\end{document}